\documentclass[11pt,a4paper]{amsart}
\usepackage{amssymb,latexsym}
\usepackage{graphicx}
\usepackage{hyperref}
\usepackage{amsmath}
\usepackage{amsfonts}
\usepackage{amssymb}
\usepackage{float}
\usepackage{wrapfig}
\usepackage{lscape}
\usepackage{rotating}
\pdfoutput=1
\newtheorem{theorem}{Theorem}[section]
\newtheorem{proposition}[theorem]{Proposition}
\newtheorem{corollary}[theorem]{Corollary}

\newtheorem{preremark}{Remark}
\newenvironment{remark}    {\begin{preremark}\rm}{\end{preremark}}
\newtheorem{lemma}[theorem]{Lemma}
\newtheorem{preexample}{Example}

\theoremstyle{definition}
\theoremstyle{remark}

\newcommand{\iso}{\cong}

\newcommand{\RR}{{\mathbb R}}

\newcommand{\ZZ}{{\mathbb Z}}

 \newcommand{\Aff}{\mathbb A}

\renewcommand{\to}{\rightarrow}

\newcommand{\HF}{\operatorname{HF}}

\newcommand\depth{\operatorname{depth}}

\newcommand\Ext{\operatorname{Ext}}

\newenvironment{amssidewaysfigure}
  {\begin{sidewaysfigure}\vspace*{.5\textwidth}\begin{minipage}{\textheight}\centering}
  {\end{minipage}\end{sidewaysfigure}}

\begin{document}
\title[WLP and stellar subdivisions]{Weak Lefschetz Property and stellar subdivisions
     of Gorenstein complexes}

\author{Janko B\"ohm}
\address{Janko B\"ohm\\
Department of Mathematics\\
University of Kaiserslautern\\
Erwin - Schr\"odinger - Str.\\
67663 Kaiserslautern\\
Germany}
\email{boehm@mathematik.uni-kl.de}

\author[Stavros~A.~Papadakis]{Stavros~Argyrios~Papadakis}
\address{Stavros~Argyrios~Papadakis\\
Department of Mathematics \\
University of Ioannina\\
Ioannina, 45110  \\
Greece}
\email{spapadak@cc.uoi.gr}

\thanks{We  thank  Christos Athanasiades for suggesting the problem and
   David Eisenbud,  Satoshi Murai and Eran Nevo for useful discussions. 
   We benefited from experiments with the computer algebra program 
  Macaulay2 \cite{GS}. J. B. was supported by DFG (German Research 
Foundation) through Grant BO3330/1-1.
 For different parts of the project S.P. was financially supported by
  FCT, Portugal and   RIMS, Kyoto University, Japan.}


\begin {abstract}
Assume $\sigma$ is a face of a Gorenstein$*$ simplicial complex $D$. We investigate
the question of whether the Weak Lefschetz Property of 
the Stanley--Reisner ring $k[D]$
(over an infinite field $k$) is equivalent to the same property of the Stanley--Reisner
ring $k[D_{\sigma}]$ of the stellar subdivision $D_{\sigma}$. We prove that
this is the case if the dimension of $\sigma$ is big compared to the codimension.
\end {abstract} 

\maketitle

\tableofcontents

\section {Introduction}

An important open question in algebraic combinatorics is whether for a 
simplicial sphere, or more generally for a Gorenstein$*$ simplicial complex,
the f-vector satisfies McMullen's g-conjecture. For details see for example 
  \cite {St1} or  \cite [Section~5.6]{Fu}.
It is well-known that for the $g$-conjecture to hold  it is enough
to prove that the Stanley--Reisner ring $\RR[D]$  	of $D$ over the real numbers
satisfies the
Weak Lefschetz Property (WLP for short). Actually Richard Stanley 
\cite{St3} proved
that if $D$ is the boundary complex of a convex simplicial polytope
it holds that $\RR[D]$ satisfies the even stronger 
Strong Lefschetz Property (SLP for short).    

Eric Babson and Eran Nevo \cite{BN}  proved  that if $k$ is an
infinite field of characteristic
$0$,  $D$ is a homology sphere with  $k[D]$ SLP
and $\sigma$ is a face of $D$ with $k[L]$ SLP 
(where $L$ denotes the link of $\sigma$ in $D$) 
then  it follows 
that $k[D_{\sigma}]$ has the SLP,
where $D_{\sigma}$ denotes the stellar subdivision of
$D$ with respect to $\sigma$.
 We investigate 
similar questions for $D$ a Gorenstein* simplicial complex
and SLP replaced by WLP. Using constructions motivated by
the interpretation of stellar subdivision in terms of Kustin--Miller 
unprojections \cite{BP1},   we prove  in Corollary~\ref{cor!theiffstatement}
that if $2 (\dim \sigma) > \dim D + 1$
and $k$ denotes an infinite field then  the Stanley--Reisner ring
$k[D]$ has the WLP if and only if  $k[D_{\sigma}]$ has the WLP.
In addition, in Corollary~\ref{cor!oneDirection3Statements} we prove that
if $k[L]$ has the  SLP  and $k[D]$ has the WLP  (or more generally 
if $K[D]$ has the WLP and $k[L]$ satisfies a certain property 
we call $M_{q,p_1}$, see Theorem~\ref{thm!forstellarwlp})
then it follows that $k[D_{\sigma}]$ has the WLP.

Section~\ref{sec!general_notations} introduces the basic notations,
while Section~\ref{subst!generalLemmasNeeded}
presents a number of general lemmas we need.
Section~\ref{sec!wlpofstellars} is the core of the 
paper and provides statements and proofs of the main results.
Section~\ref{sec!furtherresults} contains
some further results and constructions which could, perhaps, prove useful in 
attacking  the problem of whether $k[D]$ WLP is equivalent 
to  $k[D_{\sigma}]$ WLP without any assumptions
on the dimension of $\sigma$.   If this equivalence was to be
proven, it would then have as corollary 
the $g$-conjecture  for the class of PL-spheres, cf. \cite[Remark~1.3.2]{BN}.

We illustrate the structure of the arguments of the paper in Figure~\ref{fig struct}. 
The four central technical lemmas are shown in red, as well as 
Corollary~\ref{cor!theiffstatement}. The results of 
Section~\ref{sec!furtherresults}, including an alternative proof of 
Corollary~\ref{cor!theiffstatement}, are depicted in grey.

\section {Notation}  \label{sec!general_notations}
\begin{amssidewaysfigure}
    \includegraphics[scale=0.51]{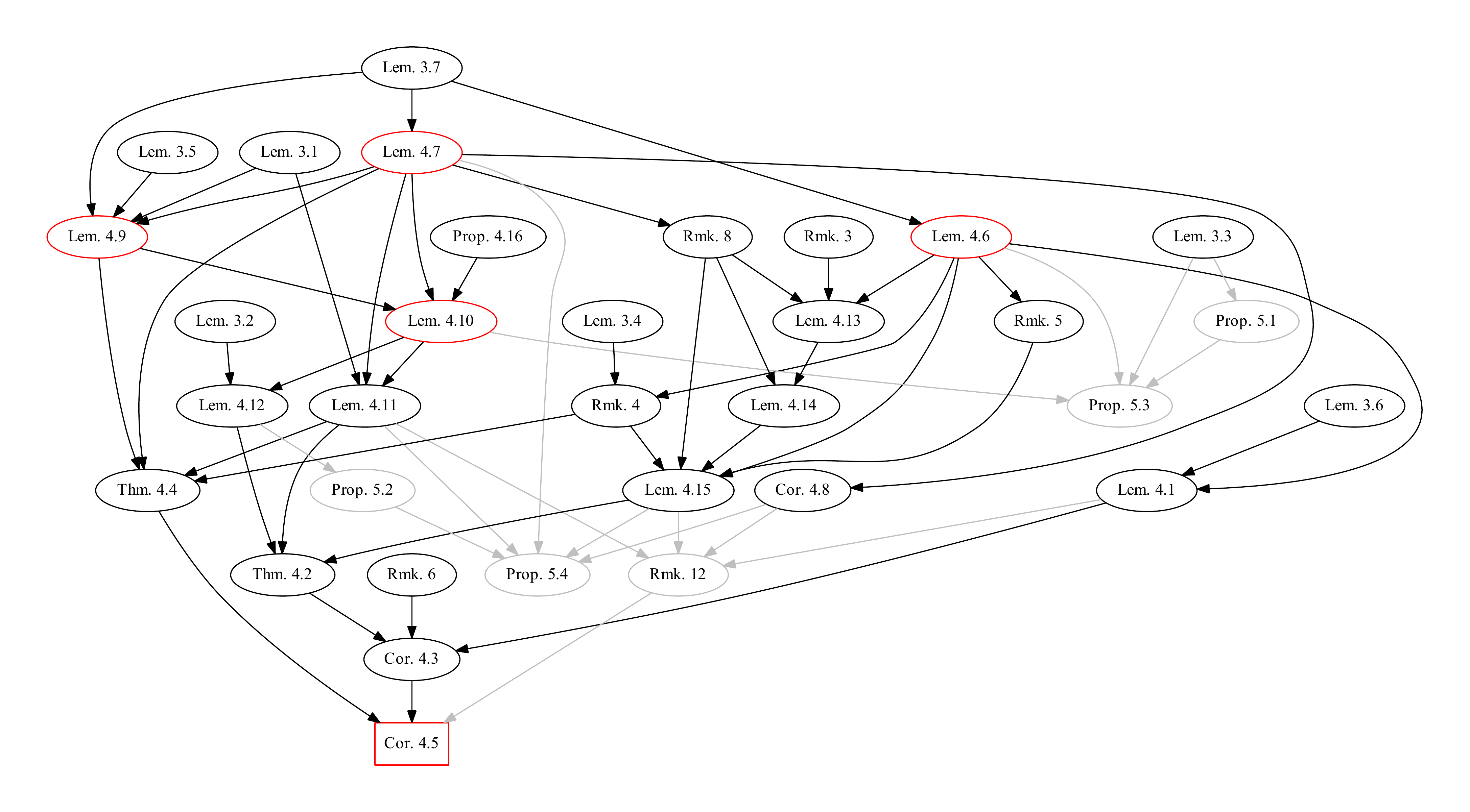}
    \caption{Structure of the arguments of the paper.}
    \label{fig struct}
\end{amssidewaysfigure}
In the following $k$ denotes an infinite field of arbitrary characteristic. 
All graded $k$-algebras will be commutative, Noetherian and of the form 
$G= \oplus_{i \geq 0} G_i$ with $G_0=k$
and  $\dim_k G_i < \infty $ for all $i$. 
The \emph{Hilbert function} of $G$
is the function $\HF(G) : \ZZ \to \ZZ$,
$ m \mapsto \dim_k G_m$. The $k$-algebra $G$
is called standard graded if it is generated,
as a $k$-algebra, by $G_1$.   An element $a \in G$ is called linear if $a \in G_1$.
For a polynomial ring we use the notions of monomial order, 
initial term, initial ideal and reverse lexicographic
order as defined in \cite [Section~15] {Ei}. 
If $G$ is a standard graded $k$-algebra with $\dim G =d$
and $f_1, \dots , f_d \in G_1$ are (Zariski) 
general linear elements of $G$ we call  $G/(f_1 , \dots , f_d)$
a general Artinian reduction  of $G$.

We say that an Artinian standard graded algebra $F$
has the \emph{Weak Lefschetz Property} (WLP for short)
if for  general $\omega \in F_1$ and all $i$ the multiplication
by $\omega$ map $F_i \to F_{i+1}$ is of maximal rank,
which means that it is injective or surjective (or both).
It is well-known (see, for example, \cite[Lemma 4.1] {BN})
that   $F$ has the WLP
if and only if there exists $a \in F_1$ such that
for  all $i$ the multiplication
by $a$ map $F_i \to F_{i+1}$ is of maximal rank. 

We say that a standard graded $k$-algebra $G$ with
$\dim G \geq 1$ has the WLP if it is Cohen--Macaulay
and for general linear elements
$f_1, \dots , f_{\dim G}$ of $G$ we have that the
algebra $G/(f_1, \dots , f_{\dim G})$, which is
Artinian by  Lemma~\ref{lem!aboutGeneralLinearsAndDepth},
has the WLP. 
Good general references for the Weak and Strong
Lefschetz Properties are
\cite {HMetal,MiNa}. 
By \cite [Proposition 3.2] {HMetal}, if $F$ is an Artinian
standard graded $k$-algebra with the WLP it follows that
$\HF(F)$ is unimodal, which means that there is no triple
$j_1 < j_2 < j_3$ such that $\HF (j_1, F) > \HF (j_2, F)$
and $\HF (j_3, F) > \HF (j_2, F)$.

We say that an Artinian standard graded algebra $F= \oplus_{i=0}^r F_i$
with $F_r \not= 0$ 
has the \emph{Strong   Lefschetz Property} (SLP for short)
if $\dim F_i = \dim F_{r-i}$ for all $i$ with $0 \leq i \leq r$ and
for a general linear element $\omega$ of $F$ and all $i$ with
$ 0 \leq 2 i \leq r$, the multiplication by $\omega^{r-2i}$ map
$F_{i} \to F_{r-i}$ is bijective.  
We say that a standard graded $k$-algebra $G$ with
$\dim G \geq 1$ has the SLP if it is Cohen--Macaulay
and for general linear elements
$f_1, \dots , f_{\dim G}$ of $G$ we have that the
algebra $G/(f_1, \dots , f_{\dim G})$, which is
Artinian by  Lemma~\ref{lem!aboutGeneralLinearsAndDepth},
has the SLP.   If $J \subset R$ is an ideal, we say that $J$ has the WLP
(resp. the SLP) if $R/J$ has the WLP  (resp. the SLP).

For a function $h : \ZZ \to \ZZ$ we define 
\[
  \Delta (h) : \ZZ \to \ZZ,  \quad \quad     m  \mapsto  h(m)- h(m-1).
\]
For $q > 0$ we inductively define $\Delta^q (h): \ZZ \to \ZZ$ by 
$\Delta^1 (h) = \Delta (h)$ and $\Delta^q (h) = \Delta^{q-1}  (\Delta(h))$
for $q > 1$.
 Assume  $G$ is a standard graded $k$-algebra 
and $a_1, a_2, \dots ,a_p$ is a regular sequence in $G$ consisting
of linear elements, then $\HF ( G/ (a_1, \dots , a_p)) = \Delta^p  (\HF(G))$.

We also define
\[
  \Delta^{+} (h) : \ZZ \to \ZZ,  \quad \quad     m  \mapsto  \max (0, h(m)- h(m-1)).
\]
Assume $F$ is an standard graded Artinian $k$-algebra. Then $F$ has the WLP if and only if 
for general $\omega \in F_1$ we have  $\HF ( F/ (\omega) ) = \Delta^{+} (\HF(F))$.
  
Assume that $h : \ZZ \to \ZZ$ has the property that there exists 
$m_0 \in \ZZ$ such that  $h(m) = 0 $ when  $m < m_0$.
We define 
\[    
          \Gamma (h) : \ZZ \to \ZZ, \quad \quad   m \mapsto  \sum_{i=-\infty}^m h(i)
\]
For $q > 0$ we inductively define $\Gamma^q (h): \ZZ \to \ZZ$ by 
  $\Gamma^1 (h) = \Gamma (h)$ and
$ \Gamma^q (h)  = \Gamma^{q-1} (\Gamma(h))$.   Assume $G$ is a standard
graded $k$-algebra  and $T_1, \dots , T_p$ are new variables of degree $1$, then
 $ \HF ( G [T_1, \dots , T_p]) = \Gamma^p ( \HF(G))$.

For a graded  $k$-algebra $G$ we denote by $\depth G$ the
depth of $G$. By \cite [Theorem 1.2.8]{BH}
\begin{equation}    \label{eqn!depthwithExts}
  \depth G = \min  \{ i:  \Ext^i_G  (k,G) \not= 0 \},
\end{equation}
where $k$ is considered as a $G$-module via
$k = G/ (\oplus_{i \geq 1} G_i)$. 
For an ideal $I$ of a ring $R$ and $u\in R$ we write 
$(I:u)=\{r\in R\bigm|ru\in I\}$ for the ideal quotient.

Assume $A$ is a finite set. We set $2^{A}$ to be the
simplex with vertex set $A$, by definition it is the set of all subsets of
$A$. A simplicial subcomplex $D\subset2^{A}$ is a subset with the property
that if $\tau\in D$ and $\sigma\subset\tau$ then $\sigma\in D$. The elements
of $D$ are also called faces of $D$, and the dimension of a face $\tau$ of $D$
is one less than the cardinality of $\tau$.  A facet of $D$ is a maximal
face of $D$ with respect to (set-theoretic) inclusion.
The dimension of $D$ is the maximum
dimension of a facet of $D$. We define the support of $D$ by
\[
  \operatorname{supp}D=\{i\in A\bigm|\{i\}\in D\}\text{.}%
\]

We denote by $R_{A}$ the polynomial ring $k[x_{a}\bigm|a\in
A]$ with the degrees of all variables $x_{a}$ equal to $1$.
For a simplicial subcomplex $D\subset2^{A}$ we define the \emph{Stanley-Reisner
ideal} $I_{D,A}\subset R_{A}$ to be the ideal generated by the square-free
monomials $\; \prod_{t=1}^{p} x_{i_t} \;$ where $\{i_{1},i_{2}%
,\ldots,i_{p}\}$ is not a face of $D$. In particular, $I_{D,A}$ contains
linear polynomials if and only if $\operatorname{supp}D\not =A$. The 
\emph{Stanley-Reisner ring} $k[D,A]$ is defined by $k[D,A]=R_{A}/I_{D,A}$. 
 For a nonempty face $\sigma$ of
$D$ we set $x_{\sigma}=\prod_{i\in\sigma}x_{i}\in R_A$. 
For a nonempty finite set $A$, we set $\partial A=2^{A}%
\setminus\{A\}\subset2^{A}$ to be the boundary complex of the simplex $2^{A}$.
In the following, when the set $A$ is clear we will
simplify the notation $k[D,A]$ to $k[D]$.

Assume that, for $i=1,2$, $D_{i}\subset2^{A_{i}}$ is a subcomplex and the
finite sets $A_{1},A_{2}$ are disjoint. By the join $D_{1}\ast
D_{2}$ of $D_{1}$ and $D_{2}$ we mean the subcomplex $D_{1}\ast D_{2}%
\subset2^{A_{1}\cup A_{2}}$ defined by
\[
D_{1}\ast D_{2}=\{\alpha_{1}\cup\alpha_{2}\bigm|\alpha_{1}\in D_{1},\alpha
_{2}\in D_{2}\}\text{.}%
\]
If $\sigma$ is a face of $D\subset2^{A}$ we define the link of $\sigma$ in $D$ to
be the subcomplex
\[
\operatorname*{lk}\nolimits_{D}\sigma=\{\alpha\in D\bigm|\alpha\cap
\sigma=\emptyset\text{ and }\alpha\cup\sigma\in D\}\subset2^{A\setminus\sigma
}\text{.}%
\]
It is clear that the Stanley-Reisner ideal of $\operatorname*{lk}%
\nolimits_{D}\sigma$ is equal to the intersection of the ideal $(I_{D,A}%
:x_{\sigma})$ with the subring $R_{A\setminus\sigma}$ of $R_{A}$. In other
words, it is the ideal of $R_{A\setminus\sigma}$ generated by the minimal
monomial generating set of $(I_{D,A}:x_{\sigma})$. 
 Furthermore, we define the
star of $\sigma$ in $D$ to be the subcomplex%
\[
\operatorname*{star}\nolimits_{D}\sigma=\{\alpha\in D\bigm|\alpha\cup\sigma\in
D\}\subset2^{A}\text{.}%
\]

If $\sigma$ is a nonempty face of $D\subset2^{A}$ and $j \notin A$,
 we define the \emph{stellar subdivision} $D_{\sigma}$ 
with new vertex $j$ to be the subcomplex
\[
D_{\sigma}=\left(  D\setminus\operatorname*{star}\nolimits_{D}\sigma\right)
\cup\left(  2^{\left\{  j\right\}  }\ast\operatorname*{lk}\nolimits_{D}%
\sigma\ast\partial\sigma\right)  \subset2^{A\cup\{j\}}\text{.}%
\]

Following \cite[p.~67]{St1}, we say that a subcomplex $D\subset2^{A}$ is
\emph{Gorenstein* over $k$} if $A=\operatorname{supp}D$, $k[D]$ is Gorenstein, and
for every $i\in A$ there exists $\sigma\in D$ with $\sigma\cup\{i\}$ not a
face of $D$. The last condition combinatorially means that $D$ is not a join
of the form $2^{\{i\}}\ast E$, and algebraically that $x_{i}$ divides at
least one element of the minimal monomial generating set of $I_{D,A}$.

Assume $D \subset 2^{A}$ is a Gorenstein* simplicial complex and 
$\sigma$ is a face of $D$.  Set
$L = \operatorname*{lk}\nolimits_{D}\sigma$.
It is well known  (cf.~ \cite[Section II.5]{St1}) that the 
subcomplex  $L \subset  2^{\operatorname{supp}L}$ is Gorenstein*
with  $\dim L = \dim D - \dim \sigma - 1$.

\section {Some general lemmas}  \label{subst!generalLemmasNeeded}

In the present section we put together  a number of general lemmas
we need.  
Of particular interest  is the following  Lemma~\ref{lem!generalLemmaForWLP}. 
It states  that under certain conditions
a non-general Artinian reduction of a WLP $k$-algebra inherits
the WLP property, and  plays a key role in the following.

\begin {lemma}  \label{lem!generalLemmaForWLP}
  Assume $k $ is an infinite field, $R=k[x_1, \dots , x_n]$ is a polynomial ring
  with all variables of degree $1$  and $T$ is a new variable of degree $1$.
   Assume $d \geq 1$, $f_1, \dots , f_d \in R_1$ are $d$ general linear elements, 
   and $J \subset R[T]$ is  a homogeneous ideal with    $\dim R[T]/J = d $.
  Assume  $R[T]/J$ has the WLP and is  Cohen--Macaulay,  and that 
   $\dim R[T]/(J + (f_1, \dots , f_d)) = 0$.  Then  
   $R[T]/(J + (f_1, \dots , f_d))$ has the WLP.
\end {lemma}

 \begin {proof}   
   Let $f_1, \dots , f_{d+1} \in R_1$ be $d+1$ general linear elements.
   For simplicity we set $H=R[T]/(J + (f_1, \dots , f_d))$.
    Since $R[T]/J$ is Cohen--Macaulay of dimension $d$  and    $\dim H = 0$
   it follows that $f_1, \dots , f_d$ is a regular sequence for $R[T]/J$. Hence
   \[
     \HF (H) = \Delta^d (\HF (R[T]/J)). 
  \]
  Since (up to a nonzero constant)  $f_{d+1}+T$ is a general linear element of $H$, 
  $H$ has the WLP if and only if $\HF (H/(f_{d+1}+T)) = \Delta^{+}  (\HF (H) )$.
  Since $f_1, \dots ,  f_{d+1}$ are general linear elements of 
  $R$, it follows that the ideal   $(f_1, \dots , f_d, f_{d+1}+T)$ 
  of $R[T]$ is equal to the ideal  of $R[T]$ generated by $d+1$ general linear
  elements of $R[T]$.    Hence,  using the assumption that $R[T]/J$ has the WLP we get
  \begin {eqnarray*}
     \HF (H/(f_{d+1}+T)) &  = &
         \HF (  R[T]/(J + (f_1, \dots , f_d ,  f_{d+1}+T)))  \\
     &  = & \Delta^{+} (\Delta^d (\HF (R[T]/J)))   \\
     &  = &  \Delta^+ ( \HF(H)).
 \end {eqnarray*}
   As a consequence, $H$ has the WLP. 
\end {proof}

\begin {remark}
   The condition  $\dim R[T]/(J + (f_1, \dots , f_d)) = 0$ 
   in the statement of Lemma~\ref{lem!generalLemmaForWLP}
   does not follow from the other assumptions 
   since the $f_i$ do not involve the variable $T$. For example, if 
   $R= k[x_1]$ and  $J = (Tx_1) \subset  R [T]$
   then   the condition is not satisfied.
\end {remark}

\begin {remark}
   Is there a statement similar to Lemma~\ref{lem!generalLemmaForWLP} for the SLP ?
\end {remark}

The following lemma is the analogue for the WLP  of  \cite [Lemma 3.3]{Mu1} which is stated for the SLP
and can be proven by the same arguments.

\begin {lemma}  \label {lem!aboutWLPofginWiebe}
   (Wiebe) Assume   $R$ is a polynomial ring over an infinite field with all
 variables of degree $1$,  $\tau$ is a monomial order on $R$ and
 $J \subset R$ is a homogeneous ideal with $R/J$ Cohen--Macaulay.
 Denote by $in_{\tau}(J)$ the initial ideal  of $J$ with respect to $\tau$.
 We assume that $R/in_{\tau}(J)$  is Cohen--Macaulay and  has the WLP.
  Then $R/J$ has the WLP.  
\end {lemma}

For a proof of the following lemma  see  \cite[Proposition~1.5.12] {BH}.  
\begin {lemma}   \label{lem!aboutGeneralLinearsAndDepth}
 Assume $k$ is an infinite field, $R = k[x_1, \dots , x_n]$ with all
variables of degree $1$,  $f_1, \dots , f_t \in R_1$ are
$t$ general linear elements of $R$ 
and $J \subset R$ is a homogeneous ideal.
If $t \leq \depth R/J$  then
$f_1, \dots , f_t$ is an $R/J$-regular sequence.
\end {lemma}

\begin{lemma}   \label{lem!wlpinthegorensteincase}
Assume   $k$ is an infinite field and  $F = \oplus_{i=0}^{d} F_i$
is an Artinian standard graded Gorenstein $k$-algebra with $F_d \not= 0$.
If $d$ is even we set $p_1 =  d/2 - 1, p_2 = d/2$, if $d$ is odd we 
set  $p_1 = p_2 = (d-1)/2$.
Denote by $\omega \in F_1$ a general linear element. Then the following are
equivalent.

i)  $F$ has the WLP. 

ii) The multiplication by $\omega$
map $F_{p_1} \to F_{p_1 + 1}$ is
injective  

iii)  The multiplication by $\omega$
map  $F_{p_2} \to F_{p_2 + 1}$ is
surjective.
\end{lemma}

\begin {proof}  It follows by \cite [Remark 2.4] {MZ}
\end {proof}

\begin {lemma}        \label {lem!3openproperties}
 Assume $k$ is an  infinite field,
 $R= k[x_1, \dots, x_n]$ with  all variables of degree $1$.
Assume $J \subset R$ is a homogeneous ideal such that $R/J$ is Cohen--Macaulay
and  $g_1, g_2 \in R$ are two nonzero linear
elements.   We define
\[
   {\mathcal {S}}_1 = \{  c \in  k :  g_1 - c g_2  \notin J  \}  \subset \Aff^1,
\]
\[
   {\mathcal {S}}_2 = \{  c \in {\mathcal {S}}_1  :  g_1 - c g_2  \; \text{ is }   R/J   \text{-regular} 
               \}  \subset \Aff^1
\] 
and
\[
   {\mathcal {S}}_3 = \{  c \in {\mathcal {S}}_2  : R/(J+(g_1 - c g_2))  \; \text{ has the WLP }                
        \}  \subset \Aff^1.
\] 
Then, for all $1 \leq i \leq 3$, the subset  ${\mathcal{S}}_i \subset \Aff^1$
is  Zariski open (but perhaps empty).
\end {lemma}

\begin {proof}  Denote by $\mathcal{B}$ the finite dimensional vector space $R_1$
considered as an affine variety.   Consider the morphism 
$\phi : \Aff^1 \to \mathcal{B}$, $c \mapsto  g_1 - cg_2$.
 It is clear that the image of $\phi$ is an affine subspace  of
 $\mathcal{B}$, hence is Zariski closed. As a consequence, it is enough to prove that,
for $1 \leq i \leq 3$, the three subsets 
\[
   {\mathcal {W}}_1 = \{  f  \in \mathcal {B}  :  f   \notin J  \}  \subset \mathcal {B},
\]
\[
    {\mathcal {W}}_2 = \{  f \in \mathcal {B}  :  f   \text { is } R/J \text{-regular}  \}
\]
 and
\[
    {\mathcal {W}}_3 = \{  f \in \mathcal {B}  :  f   \text { is } R/J\text{-regular and }  
                   R/(J+(f)) \text{ has the WLP}     \}  \subset \mathcal {B}
\]
are Zariski open.   For    ${\mathcal {W}}_1$ it is obvious. For ${\mathcal {W}}_2$ it follows
by the fact that the set of $R/J$-zero divisors is the union of the  elements of the finite set consisting
of the prime ideals associated to $J \subset R$.   The case of   ${\mathcal {W}}_3$ is also
well-known, see for example \cite[Lemma 4.1] {BN}.
\end {proof}

\begin {lemma}  \label {lem!statemnetaboutepowers}
Assume $e  \geq 1$ is an integer and  $k$ is a field of characteristic $0$
or of prime characteristic $> e$.    Consider the polynomial
ring $R = k[x_1, \dots , x_n]$ with all variables of degree $1$ and assume 
$V \subset R$ is a $k$-vector subspace.  If $a^e \in V$ for all $a \in R_1$  
 then it follows that $R_e \subset V$.
\end {lemma}

\begin {proof}
According to \cite [Section 3.2, Exercise 2]{KP},  the linear span of the  set 
$\{ a^e : a \in R_1 \}$  is equal to  $R_e$.   The result follows.  \end {proof}

\begin{lemma}   \label{lem!aboutGorensteinstar}
   Assume  $D$ is  a Gorenstein*  simplicial complex. Denote by
   $k[D]$ the Stanley--Reisner ring of $D$
   over an infinite field  $k$ and  by 
   $F$ a general Artinian reduction of $k[D]$. We have
   \[
       F = \oplus_{i=0}^{\dim k[D]} F_i
   \]
   and  $F_{\dim k[D]}$ is $1$-dimensional. 
\end {lemma}

\begin {proof}  
   Since $k$ is an infinite field and $k[D]$ is Gorenstein, hence Cohen--Macaulay, 
  by \cite [proof of Theorem 5.1.10]{BH}
  $\HF(F)$ is equal to the $h$-vector of $D$
 (for a definition of the $h$-vector of a simplicial complex see \cite[p.~205]{BH}).
  As a consequence,  \cite [Lemma~5.1.8]{BH} implies $F_i = 0$ for $i > \dim k[D]$,
  and   \cite [Lemma~5.5.4]{BH} implies that  $F_{\dim k[D]} \not= 0$,
  hence $F_{\dim k[D]}$ is $1$-dimensional by Gorenstein symmetry.
\end {proof}

\section {Weak Lefschetz Property and stellar subdivisions}   \label{sec!wlpofstellars}

The present section contains our main results. In Corollary~\ref{cor!theiffstatement}
we prove  that if  $\sigma$ is a face of a Gorenstein* simplicial complex $D$ and
$2 (\dim \sigma) > \dim D + 1$ then  the Stanley--Reisner ring
$k[D]$ has the WLP if and only if  $k[D_{\sigma}]$ has the WLP.
Moreover, in Corollary~\ref{cor!oneDirection3Statements} we prove that if
$k[\operatorname*{lk}\nolimits_{D}\sigma]$ has the  SLP  and $k[D]$ has the WLP 
then it follows that $k[D_{\sigma}]$ has the WLP.

We fix an infinite field $k$ of arbitrary characteristic
and a pair $(D,\sigma)$, where $D$ is a Gorenstein* simplicial complex 
with vertex set $\{ 1, \dots , n \}$, and
$\sigma =\{ 1,2, \dots, q+1 \}$ is a $q$-face of $D$
with $q \geq 1$.
 We set $d = \dim D +1$,  $R= k[x_1, \dots , x_n]$ 
with the degrees of all variables equal to $1$,
 $x_{\sigma} = \prod_{i=1}^{q+1} x_i  \in R$.
By  $I  \subset R$ we denote the Stanley-Reisner
ideal of $D$, hence $k[D] = R/I$ and $\dim R/I = d$.
We set $I_L = (I:x_{\sigma}) \subset R$, and
$f_1, \dots , f_{d+1}$ denote  $d+1$ general linear
elements of $R$.

Moreover, $T$ is a new variable of degree $1$ and
$J_{st} = ( I,  x_{\sigma}, T I_L)  \subset R[T]$ denotes the Stanley-Reisner ideal of
the stellar subdivision  $D_{\sigma}$,
 hence $k[D_{\sigma}] =R[T]/J_{st}$. We set 
 \[
   I_C = ( I, T^{q+1}, T I_L )  \subset R[T], \quad \quad 
   I_G = (I,  T^{q+1}- x_{\sigma},  T I_L ) \subset R[T],
\]
\[
   A = R/ (I + (f_1, \dots , f_d)),  \quad \quad    B=  R/ ( I_L + (f_1, \dots , f_d)),
\]
\[
    C = R[T]/(I_C + (f_1, \dots , f_d)),  \quad \quad  G = R[T]/(I_G + (f_1, \dots ,f_d)).
\]

The rings   $R[T]/I_C$ and $R[T]/I_G$ are closely related to  the Kustin--Miller
unprojection ring 
$S$ appearing in \cite[Theorem~1.1]{BP1}  and we view them
as  intermediate rings connecting  $k[D_{\sigma}]$ and $k[D]$.
The rings  $A,B,C$ and $G$ are
Artinian reductions of  $k[D], R/I_L, R[T]/I_C$ and $R[T]/I_G$
respectively by linears  that involve only the variables $x_i$ but are, 
otherwise,  general.

The basic properties of $A,B$ and $R/I_L$ are contained in
Lemma~\ref{lem!aboutAandB}, of $C$ are  contained  in
 Lemma~\ref{lem!structureOfCno2}, of $R[T]/I_G$ and $G$ are contained in
Lemma~\ref{lem!structureOfG} and 
of $R[T]/I_C$ are contained in Lemma~\ref{lem!aboutIC}.
Since $\dim k[D] = d$, it follows that  $A$ is a general Artinian reduction of $k[D]$.
Since by Lemma~\ref{lem!aboutAandB} $\dim R/I_L = d$, 
it follows that  $B$ is a general Artinian reduction of $R/I_L$.    
On the other hand $C$  is not a general Artinian reduction of  $R[T]/I_C$,
even though by Lemma~\ref{lem!aboutIC} $\dim R[T]/I_C = d$,
since the $f_i$ do not involve the variable $T$. 

We denote by $\; L \subset 2^{\{q+2,q+3, \dots , n \}}\;$  the link of $\sigma$ in $D$
and set $k[L] = k[L, \{q+2, q+3,\dots , n \}]$. Using that
 \begin{equation}  \label {eqn!relatingLinksDefinitions}
              R/I_L  \iso  (k[L])    [x_1, \dots , x_{q+1}]
\end {equation}
it follows that $B$ is isomorphic to a general Artinian reduction of 
$k[L]$.

\begin {remark}  \label{rem!poincare_duality_for_gor}
We will use the well-known fact,  see for example  \cite[Proposition~3.3]{MN}
or  \cite [Theorem 2.79] {HMetal},
that if $F= \oplus_{i=0}^r F_i$
with $F_r \not= 0$   is a standard
graded Gorenstein Artinian $k$-algebra then
$F_r$ is $1$-dimensional, and for all $i$ with $0 \leq i \leq r$ the
multiplication map $F_i \times F_{r-i} \to F_r \iso k$ is a perfect pairing.
We will refer to $F_{r-i}$ as the Poincar\'e  
dual of $F_i$.
As a consequence, given $i,j$ with $0 \leq i \leq j \leq r$ and
$0 \not= e \in F_i$ there exists $e' \in F_{j-i}$ such that $ee' \not= 0$ in $F_j$.
\end {remark}

If $d$ is even
we set $p_1 =  d/2 - 1, p_2 = d/2$, while if $d$ is odd we 
set  $p_1 = p_2 = (d-1)/2$.

\begin {remark}  \label{rem!aboutwlpforA}
Denote by $\omega \in A_1$ a general linear element.
Since by Lemma~\ref{lem!aboutAandB}   $A$ is Gorenstein 
with $A_i =0$ for $i \geq d+1$ and
$A_d \not= 0$,  Lemma~\ref{lem!wlpinthegorensteincase}
implies that   $A$ has the WLP if and only if the
multiplication by $\omega$
map $A_{p_1} \to A_{p_1 + 1}$ is
injective and by the same lemma this happens if and only if the
multiplication by $\omega$
map  $A_{p_2} \to A_{p_2 + 1}$ is surjective.
\end {remark}

We say that $k[L]$, or equivalently  $R/I_L$ 
or equivalently $B$,  has the  property $M_{q,p_1}$  
if  for a general element $\omega \in B_1$ the
multiplication by $\omega^q$ map  $B_{p_1-q} \to B_{p_1}$ 
is injective.   

\begin {remark}  \label {rem!aboutMqpropertyno2}
Since by Lemma~\ref{lem!aboutAandB} $B$ is Artinian Gorenstein
with $B_i =0 $ for $i \geq d-q$ and $B_{d-q-1} \not= 0$,
by Gorenstein duality  $k[L]$ has the property $M_{q,p_1}$ 
if and only if for a general element $\omega \in B_1$ the
multiplication map by $\omega^q :  B_{d-q-1-p_1} \to B_{d-q-1 -(p_1-q)}$ 
is surjective.     If $d$ is even then $d=2p_1+2$, hence
 $d-q-1-p_1 = p_1+1-q = p_2 - q$, while if $d$ is odd then
$d = 2p_1+1$, hence $d-q-1-p_1 = p_2 - q$.  Hence no
matter if $d$ is even or odd, 
$k[L]$ has the property $M_{q,p_1}$ 
if and only if 
for a general element $\omega \in B_1$ the
multiplication by $\omega^q$ map $B_{p_2-q} \to B_{p_2}$ 
is surjective.   
\end {remark}

\begin {remark}  \label {rem!about_Mq_peoperty}
Assume $d$ is odd.  Since  $d = 2p_1+1$, we have
$d-q-1-(p_1-q)= d-p_1-1=p_1$, hence the multiplication
by $\omega^q$ map in the definition of property $M_{q,p_1}$  is
between Poincar\'e dual graded components of $B$.  
Assume $d$ is even. We have $d = 2p_1+2$ and
$d-q-1-(p_1-q)= d-p_1-1=p_1+1$, hence 
the multiplication by $\omega^q$ map  in the definition 
of the property $M_{q,p_1}$  factors as  $f \circ g$, where
$g$  is the  multiplication 
by $\omega$  map $B_{p_1-q} \to B_{p_1-q+1}$ 
and $f$ is multiplication by $\omega^{q-1}$ map
between the  Poincar\'e  duals  $B_{p_1-q+1}$ and $B_{p_1}$.
As a consequence, no matter if $d$ is even or odd
if $B$ has the SLP  then it follows that  
property $M_{q,p_1}$ holds for $B$.
\end {remark}

\begin {lemma}  \label {lem!aboutMqforhighq}
   If  ($q > p_1$)  or  ($q = p_1$ and the field
 $k$ has characteristic $0$ or a prime number $>  d-q-1$)
  then property $M_{q,p_1}$ holds for $k[L]$.
\end {lemma}

\begin {proof} If $q > p_1 $ then $B_{p_1-q} = 0$ and the
result is obvious.

Assume  $q = p_1$.  Since $B_0= k$,  property $M_{q,p_1}$ for $k[L]$ 
is equivalent to $\omega^{p_1} \not= 0$ in $B$ for general $\omega \in B_1$.
To get a contradiction we assume this property is not true, then  
it follows that $b^{p_1} = 0$  for all $b \in B_1$.  
This, together with the assumptions on the characteristic of the field $k$ 
imply by the general Lemma~\ref{lem!statemnetaboutepowers}
that  $B_{p_1} =0$. Since $d-q-1 = d-p_1 -1 \geq p_1$,   
and $B$ is standard graded we get $B_{d-q-1} = 0$. By Lemma~\ref{lem!aboutAandB}  
$B_{d-q-1} \not= 0$  which is a contradiction. 
\end {proof}

\begin {theorem}  \label{thm!forstellarwlp}   Assume $k[D]$ has the WLP and
property $M_{q,p_1}$ holds for $k[L]$.  Then $k[D_{\sigma}]$ has the WLP.
\end {theorem}

\begin {proof}
   By  Lemma~\ref{lem!technicallemma1}    $A$ has the WLP.
   Hence by Lemma~\ref{lem!technicallemma3} $C$ has the WLP.
   As a consequence, by  Lemma~\ref{lem!technicallemma1}    
   $R[T]/I_C$ has the WLP.  Using Lemma~\ref{lem!technicallemma2}
   the result follows.
\end {proof}

\begin {corollary}  \label{cor!oneDirection3Statements}
  i)    Assume  $k[L]$  has the SLP. Then $k[D]$
WLP  implies that   $k[D_{\sigma}]$ has the WLP.
   
  ii)  Assume $q > p_1 $ or  ($q =p_1$ and the field
      $k$ has characteristic $0$ or a prime number 
      $ > d-q-1$).  Then $k[D]$ WLP implies that  $k[D_{\sigma}]$ has the WLP.
\end {corollary}

\begin {proof}   Part i)  follows from Theorem~\ref{thm!forstellarwlp},
 since by Remark~\ref{rem!about_Mq_peoperty}
$k[L]$  SLP  implies that
the property $M_{q,p_1}$ holds for  $k[L]$.

We prove Part ii). Using Lemma~\ref{lem!aboutMqforhighq}   property $M_{q,p_1}$ holds for $k[L]$.
Hence  $k[D_{\sigma}]$ has the WLP by Theorem~\ref{thm!forstellarwlp}.
\end {proof}

The proof of the following theorem will be given in Subsection~\ref{subs!proofOfTHmforsDwlp}.

\begin {theorem}  \label{thm!forsDwlp}  Assume $q > p_2 $  and $k[D_{\sigma}]$ has the WLP.
Then $k[D]$ has the WLP.  
\end {theorem}

\begin {corollary}  \label{cor!theiffstatement}
 Assume $2 (\dim \sigma) > \dim D +1$. Then $k[D]$ has the WLP if and only if 
       $k[D_{\sigma}]$  has the WLP.  
\end {corollary}

\begin {proof}
   We first prove that the statement   $2 (\dim \sigma) > \dim D +1$ is equivalent to
      $q > p_2$. Indeed, by the definitions   $q = \dim \sigma$, $d = \dim D +1$.  
       Assume  $d$ is even.      Then  $ p_2 = d/2$.   Hence $q > p_2$
             is equivalent to $\dim \sigma > d/2  $  which is equivalent
                               to $ 2 (\dim \sigma) > d= \dim D +1$.
            Assume now $d$ is odd.       Then  $ p_2 = (d-1)/2$,  hence $q > p_2$
             is equivalent to $\dim \sigma > (d-1)/2  $  which is equivalent
                               to $ 2 (\dim \sigma) > d-1$.
             But $d$ odd implies $d-1$ even, hence  since $ 2 (\dim \sigma)$ is always even
                   $ 2 (\dim \sigma) > d-1$ is equivalent to
                   $ 2 (\dim \sigma) > d = \dim D +1 $.  

     Assume  $2 (\dim \sigma) > \dim D +1$  and $k[D]$ has the WLP.   As we said above
     $q > p_2$.     Since $p_2 \geq p_1$, we have $q > p_1$, hence Part ii) of
     Corollary~\ref{cor!oneDirection3Statements}  implies that  $k[D_{\sigma}]$  has the WLP.  

    Assume now   $2 (\dim \sigma) > \dim D +1$  and $k[D_\sigma]$ has the WLP.   As we said above
     $q > p_2$.  By Theorem~\ref{thm!forsDwlp}  $k[D]$ has the WLP.
\end {proof}

\begin{remark}    For a second proof  of  Corollary~ \ref{cor!theiffstatement}   see
Remark~\ref {rem!secondproofofThmforhighq}.
\end {remark}

\begin {lemma}  \label{lem!aboutAandB}
   The rings  $k[D], k[D_{\sigma}], A$ and $B$ are Gorenstein.  
   $A= \oplus_{i=0}^d A_i$ with $A_d$ $1$-dimensional.
   $R/I_L$ is Gorenstein with $\dim R/I_L=d$. We have 
  $B= \oplus_{i=0}^{d-q-1} B_i$ and $B_{d-q-1}$  is $1$-dimensional.
   Moreover, for all $m \geq 0$ we have 
  \[
        \HF( m, k[D_{\sigma}])  =  \HF( m, k[D])   +
                         \sum_{i=1}^q   \HF( m-i, R/I_L).
   \]
\end {lemma}

\begin {proof}  Since by assumption  $D$ is Gorenstein*, it follows that $k[D]$ is Gorenstein.
    By definition $A$ is a general Artinian reduction of $k[D]$, hence it is also Gorenstein.
   Since $\dim k[D]=d$,
    using Lemma~\ref{lem!aboutGorensteinstar}  we get 
   $A= \oplus_{i=0}^d A_i$ with $A_d$ $1$-dimensional.

  By \cite[p.~188]{St2} Gorenstein* is a topological property.
  Hence a  stellar subdivision of a Gorenstein*  simplicial complex 
  is again Gorenstein*. As a consequence
  $k[D_{\sigma}]$  is Gorenstein.

  As already mentioned in Section~\ref{sec!general_notations},
  we have that $L \subset 2^{\operatorname{supp}L}$ is
  Gorenstein*,  with  $\dim k[L] = d-q-1$.
   Using  Equation~(\ref {eqn!relatingLinksDefinitions})
  it follows that $R/I_L$ is Gorenstein of dimension $d$.
  Since $B$ is isomorphic to a general Artinian reduction of $k[L]$,
  it follows that $B$ is Gorenstein.  
  Moreover, Lemma~\ref{lem!aboutGorensteinstar}  implies that
  $B= \oplus_{i=0}^{d-q-1}  B_i$ and $B_{d-q-1}$ is $1$-dimensional.

   The equation between the Hilbert functions follows from \cite[Remark~5]{BP1}.
\end {proof}

\begin {lemma}  \label{lem!structureOfCno2}
   (Recall $\sigma$ is a $q$-face, with $q \geq 1$)
There is a well-defined bijective $k$-linear map of vector spaces
\[
   A \oplus B^q  \to C,  \quad  ([a], [b_1], \dots , [b_q]) \mapsto
          [a + \sum_{i=1}^q b_i T^i ]       
\]
for   $a,b_i \in R$.  
As a corollary, $C$ is Artinian and for all $m \geq 0$
 \[
     \HF (m, C) = \HF (m, A) + \sum_{i=1}^q \HF (m-i, B).
\]
Hence, $\HF(C)$ is equal to the Hilbert function of
a general Artinian reduction of $k[D_{\sigma}]$.
In  particular $C_{i} = 0$ for $i \geq d+1$ and $C_{d}$ is $1$-dimensional.
\end {lemma}

\begin {proof}
 By definition $C= R[T]/ (I, TI_L, T^{q+1}, f_1, \dots , f_d)$. We also have
$A= R/(I, f_1, \dots , f_d)$ and $B = R/( I_L, f_1, \dots , f_d)$. 
Denote by $\phi$ the map in the statement of the present proposition.

$\phi$ is well-defined: Assume  $a,a',b_i,b_i' \in R$ with
$[a] = [a']$ in $A$ and $[b_i] = [b_i']$ in $B$ for all $i$.
Then $a-a' \in (I, f_1, \dots , f_d)$ and
$b_i - b_i' \in (I_L, f_1, \dots , f_d)$, hence $T(b_i-b_i') \in
(TI_L, f_1, \dots , f_d)$. As a consequence
 \[
           [a + \sum_{i=1}^q b_i T^i ]      =           [a' + \sum_{i=1}^q b_i' T^i ]       
\]
in $C$.

$\phi$ is surjective: Obvious from the definitions of $\phi$ and $C$.

$\phi$ is injective:    Assume $a,b_i \in R$ with 
$[a+ \sum_{i=1}^q b_i T^i] = 0 $ in $C$. This implies that
there exist  
   $e_{a,1},  \dots , e_{a,r_1} \in I$,
   $e_{b,1},  \dots , e_{b,r_2} \in I_L$,
   $g_{a,1}, \dots , g_{a,r_1} \in  R[T]$,
   $g_{b,1}, \dots , g_{b,r_2} \in  R[T]$,
   $g_c \in R[T]$,  $ g_{e,1}, \dots ,g_{e,d} \in R[T]$
such that
\[
    a + \sum_{i=1}^q b_i T^i  = 
          \sum_{j=1}^{r_1} g_{a,j}e_{a,j} + 
         T \sum_{j=1}^{r_2} g_{b,j}e_{b,j} +
          g_c T^{q+1} +
         \sum_{j=1}^d  g_{e,j}f_j
\]
with equality in $R[T]$. Looking at the coefficients of
the powers of $T$ we get $a \in (I, f_1, \dots , f_d)$ and
$b _i \in  ( I_L, f_1, \dots , f_d)$  for all $1 \leq i \leq q$. 
Hence $\phi$ is injective.

Since $A,B$ are Artinian,  they are finite dimensional  
$k$-vector spaces. Since $\phi$ is surjective  $C$ is finite dimensional
as a $k$-vector space which implies that $C$ is Artinian.

The formula connecting the Hilbert functions of $A,B,C$
is an immediate consequence of the fact that $\phi$ is bijective.

Using \cite[Remark~5]{BP1}
it follows that   $\HF(C)$ is equal to the Hilbert function of
a general Artinian reduction of $k[D_{\sigma}]$.
As a consequence,
the statements  $C_{i} = 0$ for $i \geq d+1$ and $C_{d}$ is $1$-dimensional
follow from Lemma~\ref{lem!aboutGorensteinstar} applied to the Gorenstein*
simplicial complex $D_{\sigma}$.    \end {proof}

\begin{remark}    \label {rem!structureOfCSecondVersion}
Taking graded components, Lemma~\ref{lem!structureOfCno2} immediately implies 
that,  for  $i \geq 0$, there exists a $k$-vector space decomposition
\[
     C_i = A_i \oplus ( \oplus_{j=1}^q  B_{i-j}T^j ).
\]
Hence, if $c \in C_i$ there exist unique $a \in A_i$ and 
$b_j \in B_j$ such that  
\[
     c = a + b_{i-1}T + b_{i-2}T^2 + \dots + b_{i-q} T^q.
\]
\end {remark}

\begin {corollary}   \label{corol!HFofsectionofC}
  Assume  that the Hilbert function of a general linear section of $C$ is
  equal to the Hilbert function of a general linear section  of a general
 Artinian reduction of $k[D_{\sigma}]$. Then $C$ has the WLP if and only if
 $k[D_{\sigma}]$ has the WLP.
\end {corollary}

\begin {proof}   By Lemma~ \ref{lem!structureOfCno2}
  $\HF(C)$ is equal to the Hilbert function of
a general Artinian reduction of $k[D_{\sigma}]$. The result follows
 from the definition of WLP.
\end {proof}

\noindent Recall  $I_G =(I,  T^{q+1}- x_{\sigma},TI_L)$  and  $G = R[T]/(I_G+(f_1, \dots , f_d))$.

\begin {lemma}  \label{lem!structureOfG}
i)  The $k$-algebra $R[T]/I_G$ is Gorenstein, $\dim R[T]/I_G = d$ and 
$\HF(R[T]/I_G) = \HF(k[D_{\sigma}])$.

ii)  The ring $G$ is Artinian Gorenstein, and 
  $\HF (G) = \HF (C)$, which by Lemma~\ref {lem!structureOfCno2} is equal to the HF
 of a general Artinian reduction of $k[D_{\sigma}]$.

iii)  Assume $k[D_{\sigma}]$ has the WLP. Then both $R[T]/I_G$ and $G$ 
have the WLP.
\end {lemma}

\begin {proof}  Let $z,z'$ be two new variables and $c \in k$. We 
  set  $I_V = (I, T I_L, Tz- x_{\sigma}) \subset R[T,z]$ 
  where $\deg T = 1$, $\deg z = q$.
  We set $\mathcal {M} =  R[T,z,z']/(I_V)$, where
  $(I_V)$ is the ideal of  $R[T,z,z']$  generated by  $I_V$
  and $\deg T = \deg z' =  1$,  $\deg z =q $.  
  We also set  $\mathcal {Q} =  \mathcal{M}/(z - T^{q-1}z')$ and
  \[ 
      H_c =      \mathcal {Q}/  (z' - c^{q+1}T ) = R[T]/(I,  TI_L, c^{q+1}T^{q+1}- x_{\sigma}).
  \]

 By \cite[Theorem~1.1]{BP1}  $R[T,z]/I_V$ is Gorenstein and 
  $\dim R[T,z]/I_V = d+1$.
  It follows that  $\mathcal{M}$ is Gorenstein and
  $\dim \mathcal {M}  = d+2$.  Hence $\dim \mathcal {Q} \geq d+1$.
  Since $\mathcal{Q}/ (z') =  k[D_{\sigma}]$ which has dimension $d$, it follows
  that $\dim \mathcal {Q}  \leq d+1$. Hence $\dim \mathcal {Q} = d+1$.
   Using that $\mathcal{M}$ is Gorenstein, hence Cohen--Macaulay,
   it follows that $z - T^{q-1}z'$ is an 
   $\mathcal{M}$-regular element, hence $ \mathcal {Q}$ is Gorenstein.
   
Clearly 
  \[
     \mathcal {Q}  =  R[T,z']/ (I, T I_L, T^q z' - x_{\sigma}),
  \]
   hence $\mathcal {Q}$ is standard graded.  
  We have $H_0 = k[D_{\sigma}]$, hence  $\dim  \mathcal {Q}/ (z') =
     \dim \mathcal {Q} -1$. Since $\mathcal {Q}$ is Gorenstein, 
   hence Cohen--Macaulay,  it follows that $z'$ is  a
  $\mathcal {Q}$-regular element.

  Hence Lemma~\ref{lem!3openproperties} implies that
  for general $c \in k$ we have that $z'-c^{q+1}T$ is
  a $\mathcal {Q}$-regular element, since the property is true 
  for $c=0$.   As a consequence, for general $c \in k$ the ring
  $H_c$   is Gorenstein of dimension $d$ and $\HF(H_c) = \HF (k[D_{\sigma}])$.
 For nonzero $c$ using the linear change of coordinates $T \mapsto cT$ we can assume
 that $c = 1$.  Since $R[T]/I_G$ is isomorphic to  $H_1$, it follows that
 $R[T]/I_G$ is Gorenstein,  $\dim R[T]/I_G= d$ and 
$\HF(R[T]/I_G) = \HF(k[D_{\sigma}])$.

We now prove ii) We first prove that $G$ is Artinian.
The arguments used in the proof of Lemma~\ref{lem!structureOfCno2} also
give that there exists a well-defined surjective $k$-linear map of vector spaces
\[
  \psi :  A \oplus B^q  \to G,  \quad  ([a], [b_1], \dots , [b_q]) \mapsto
          [a + \sum_{i=1}^q b_i T^i ]       
\]
for   $a,b_i \in R$. Since $A,B$ have finite dimension as $k$-vector spaces
it follows that $G$ has finite dimension as a $k$-vector space, hence
it is Artinian.
Since we proved that $R[T]/I_G$ is Gorenstein,
hence Cohen--Macaulay, and of dimension $d$, 
$G$ Artinian implies that $f_1, \dots , f_d$ is a regular sequence
for $R[T]/I_G$. As a consequence, using that  we proved above that 
$\HF(R[T]/I_G) = \HF(k[D_{\sigma}])$ it follows that
\[
    \HF(G)  =\Delta^d (\HF(R[T]/I_G))  = \Delta^d (\HF(k[D_{\sigma}])  = \HF (C)
\]
with the last equality by Lemma~\ref{lem!structureOfCno2}. Hence 
$\dim_k G = \dim_k C$.
Using that by Lemma~\ref{lem!structureOfCno2}   
    $\dim_k C = (\dim_k A) + q (\dim_k B)$
we get that      $\dim_k G = (\dim_k A) + q (\dim_k B)$.
Since $\psi$ is surjective it follows that $\psi$ is bijective.

  We now prove iii). Assume $k[D_{\sigma}]$ has the WLP.  By  Lemma~\ref{lem!3openproperties}
 we have that   for general  $c \in k$ the algebra $H_c$ 
 has the WLP, since the property is true for $c=0$.  
 For nonzero $c$ using the linear change of coordinates $T \mapsto cT$ 
 we can assume  that $c = 1$. Hence $R[T]/I_G$ has the WLP.  Since by ii) $G$ is Artinian, 
 Lemma~\ref{lem!generalLemmaForWLP} implies 
 that $G$ has the WLP.
\end {proof}

Due to its length the proof of the following lemma will be given in Subsection~\ref{subs!proofoflemmaaboutIP}.

\begin {lemma}  \label{lem!aboutIC}
  The ring   $R[T]/I_C$ is Cohen--Macaulay with 
$\dim R[T]/I_C = d$ and $\HF(R[T]/I_C) = \HF ( k [D_{\sigma}])$.   
Moreover,  $I_C$ is 
the initial ideal of  $I_G$ with respect to the reverse lexicographic order
in $R[T]$ with $T > x_1 >\dots > x_n$. 
\end {lemma}

\begin {lemma}  \label{lem!technicallemma1}
   $k[D]$  (resp. $k[L]$, resp. $R[T]/I_C$)  has the WLP
    if and only if $A$ (resp. $B$, resp. $C$) has the WLP.
\end {lemma}

\begin {proof}
  Since  $A$ is a general Artinian reduction of the Gorenstein $k[D]$ 
  it is immediate that $A$ has the WLP if and only if $k[D]$ has the WLP.
  Since  $B$ is a general Artinian reduction of the Gorenstein $k[L]$ 
  it is immediate that $B$ has the WLP if and only if $k[L]$ has the WLP.

   By Lemma~\ref {lem!aboutIC}  $R[T]/I_C$ is Cohen--Macaulay of dimension $d$. 
   Since by Lemma~\ref{lem!structureOfCno2}  
$C$ is Artinian, it follows that $f_1, \dots , f_d$ is an $R[T]/I_C$-regular sequence.
   Hence $C$ WLP implies $R[T]/I_C$ WLP.  Conversely, assume that $R[T]/I_C$ has the WLP. 
The result that
   $C$ has the WLP follows by Lemma~\ref{lem!generalLemmaForWLP}.
\end {proof}     

\begin {lemma}  \label{lem!technicallemma2}
 Assume   $R[T]/I_C$ has the WLP. Then  $k[D_{\sigma}]$ has the WLP. 
 \end {lemma}

\begin {proof}
    We argue in a very similar way to   \cite [Proposition~2.2]  {Mu1}. 
     Assume $R[T]/I_C$ has the WLP. 
    The ordering we use in the polynomial ring $R[T]$ is the reverse lexicographic 
    ordering with  $T > x_1 > x_2 > \dots > x_n$.

   Consider the $k$-algebra
    automorphism $\phi$ of $R[T]$ defined by $T \mapsto T, x_i \mapsto x_i+T$ for 
   $1 \leq i \leq q+1$   and $x_i \mapsto x_i$ for $q+2 \leq i \leq n$. 
    We claim that $I_C$ is the initial  ideal
    of $\phi (J_{st})$. Indeed, it is clear that $I_C$ is a subset of 
    the initial  ideal    of $\phi (J_{st})$ and by 
    Lemma~\ref{lem!aboutIC} $\HF(R[T]/I_C) = \HF ( k [D_{\sigma}])$.   
    It follows by Lemma~\ref {lem!aboutWLPofginWiebe} that 
    $R[T]/\phi(J_{st})$ 
     has the WLP, hence also $R[T]/J_{st} = k[D_{\sigma}]$ has the WLP. 
\end {proof}

The following lemma will be used in the proof of
 Lemma~\ref{lem!technicallemma3}.
 Since $C$ is not Gorenstein, it does not follow
 from   Remark~\ref{rem!poincare_duality_for_gor}. 

\begin {lemma}  \label{lem!technicallemma4}
    Assume $1 \leq i < p_1$ and $0 \not= c \in C_i$. Then there exists
  $c'  \in C_{p_1 -i}$ such that  $cc'  \not= 0$ in $C_{p_1}$. 
    As a corollary, assume  $\omega \in C_1$ is any element, not
necessarily general.
  If the multiplication by $\omega$ map $C_{p_1} \to C_{p_1+1}$
 is injective, it follows that for all $i$ with $1 \leq i  \leq p_1$
 the multiplication by $\omega$ map $C_i \to C_{i+1}$ is injective.
\end {lemma}

\begin {proof}  
Using  Remark~\ref{rem!structureOfCSecondVersion}  write 
   \[ 
       c = a_i + \sum_{j=1}^q b_{i-j} T^j 
  \]
  (equality in $C$)  with $a_i \in A_i$ and $b_{i-j} \in B_{i-j}$ for all $1 \leq j \leq q$.
  Since $c \not= 0$, we have $a_i \not= 0$ or $b_{i-e} \not= 0$ for some $e$
  with $1 \leq e \leq \min \{i,q\} $.

  Assume first that $a_i \not= 0$.  By Lemma~\ref{lem!aboutAandB}
  $A$ is Artinian Gorenstein
  with $A_{d} \not= 0$.  
  Hence by  Remark~\ref{rem!poincare_duality_for_gor}  
  there exists  $a \in A_{p_1-i}$ such that
  $a_i a \not= 0  \in A_{p_1}$.  
  Hence $cc'  \not= 0$ in $C_{p_1}$, where $c' = a$.

  For the rest of the argument we  assume that there
  exists $e$ with $1 \leq e \leq \min \{i,q\}$ such that $b_{i-e} \not= 0$ in $B$.
  We have two cases:

  First Case:  We assume  $p_1 \leq d-q-1$. 
   By   Lemma~\ref{lem!aboutAandB}  $B$ is Artinian Gorenstein
  with $B_{d-q-1} \not= 0$.
   By  Remark~\ref{rem!poincare_duality_for_gor}  
  there exists  $b \in R_{p_1-i}$ such that $b_{i-e} b  \not= 0$ in $B_{p_1-e}$. 
  Hence $cc'  \not= 0$ in $C_{p_1}$, where $c' = b$.

    Second Case:  We assume $d-q-1 < p_1$.   
    Hence  $d-p_1-1 < q$.  
    If  $d$ is even, since $d =2p_1+ 2$ we get $p_1 + 1 < q$.
    If  $d$ is odd, since $d =2p_1+ 1$ we get $p_1  < q$.
    Hence no matter if $d$ is even or odd  we have  $ p_1 < q$.
     Since $e \leq i$ and $i < p_1$ we get     $0 < p_1 - i < q$ and
     $0 < p_1 + e -i  < q$. As a consequence 
     $ 0 \not=        b_{i-e}T^{e}   T^{p_1-i}$ in $C_{p_1}$.
    Hence $cc'  \not= 0$ in $C_{p_1}$, where $c' = T^{p_1-i}$.
 
   We now prove the corollary. We assume  $1 \leq i < p_1$,  
   that the multiplication by $\omega$ map $C_{p_1} \to C_{p_1+1}$
   is injective and  that 
   the multiplication by $\omega$ map $C_i \to C_{i+1}$ is  not
   injective and we will get a contradiction. By the assumptions
    there exists $0 \not= c \in C_i$ such that $\omega c =0$
    in $C_{i+1}$.   By the first part of the present lemma there  exists
    $c' \in C_{p_1 -i}$ such that  $cc' \not= 0$ in
    $C_{p_1}$. Hence by the assumptions $ \omega cc' \not= 0$ in $C_{p_1+1}$,
   which contradicts $\omega c = 0$ in $C_{i+1}$.
\end {proof}

The ring $C$ is not Gorenstein, hence we can not use
Lemma~\ref{lem!wlpinthegorensteincase}.  
The following lemma is a substitute.

\begin {lemma}   \label{lem!technical_lemma_57}
 The following are equivalent

i)  $C$ has the WLP.

ii)  For general $\omega \in R_1$ the multiplication by $\omega +T$ map
    $C_{p_1} \to C_{p_1+1}$ is injective and the multiplication 
   by $\omega +T$ map $C_{p_2} \to C_{p_2+1}$
    is surjective.
\end {lemma}

\begin {proof} 
Assume that i) holds. 
 Since for nonzero $c \in k$ the map
$C \to C$, with $x_i \mapsto x_i$ and $T \mapsto cT$ is
well-defined and an automorphism, it follows that
for general $\omega \in R_1$ we have that
$\omega + T$ is a general element of $C_1$.
Since $C$ is assumed to have the  WLP  to prove ii) it is enough to 
prove $\dim C_{p_1} \leq \dim C_{p_1+1}$ and
$\dim C_{p_2 +1} \leq \dim C_{p_2}$. 
We assume it is not the case and we will get a contradiction.

 By \cite [Proposition 3.2] {HMetal}
$\HF(C)$ is unimodal. 
Using  Remark~\ref {rem!structureOfCSecondVersion} 
it follows that  $\dim C_{i} = \dim C_{d-i}$ for all $i \in \ZZ$.
First we assume that  $\dim C_{p_1} > \dim C_{p_1+1}$.
If $d$ is odd then $\dim C_{p_1} = \dim C_{p_1+1}$ which is a
contradiction. Hence $d$ is even. Since $d-p_1 = p_1+2$ it follows that
$\dim C_{p_1+2} = \dim C_{p_1} > \dim C_{p_1+1}$, which contradicts
the unimodality of $\HF(C)$.

We now assume that $\dim C_{p_2 +1} > \dim C_{p_2}$. 
If $d$ is odd  then $p_1 = p_2$, and hence  $\dim C_{p_2} = \dim C_{p_2+1}$,
which is a contradiction.
If $d$ is even $p_2=  p_1+1$ and $d-(p_2+1) = p_1$. Hence
$\dim C_{p_1} > \dim C_{p_1+1}$ which we proved
above that can not happen.

Conversely  assume that ii) holds.  
 Then by  Lemma~\ref{lem!technicallemma4} for all $0 \leq i \leq p_1$
the multiplication by $\omega +T$ map $C_i \to C_{i+1}$ is injective. 
Since $C$ is standard graded,  the assumption that the multiplication 
   by $\omega +T$ map $C_{p_2} \to C_{p_2+1}$
    is surjective implies
that for all $j \geq p_2$ the multiplication
by $\omega +T$ map $C_j \to C_{j+1}$ is surjective.  Since
$p_2 = p_1$ if $d$ is odd and  $p_2 = p_1 + 1$ if $d$ is even, we get
that for all $j \in \ZZ$ multiplication
 by $\omega +T$ as a map $C_j \to C_{j+1}$ is  injective or surjective (or both),
hence $C$ has the WLP.
\end {proof}

Due to its length the proof of the following lemma will be given
in Subsection~\ref{subs!proofoflemma3}.

\begin {lemma}  \label{lem!technicallemma3}
   The following are equivalent:

    i)   $C$ has the WLP. 

   ii)  $A$ has the  WLP and  property $M_{q,p_1}$ holds for $B$.
\end {lemma}

\subsection {Proof of Lemma~\ref{lem!aboutIC}} \label {subs!proofoflemmaaboutIP}

For the proof of Lemma~\ref{lem!aboutIC} we need the following proposition.

\begin {proposition}  \label{prop!anputmathcalH}
  Set  $\mathcal {H} = k[x_1, \dots , x_n, T, z]/(I,Tz, TI_L)$. Then
   $\mathcal {H}$ is Cohen--Macaulay and $\dim \mathcal {H} =  d+1$.
\end {proposition}

\begin {proof}  Recall that $*$ denotes the join of simplicial complexes and
 for a finite set $S$ we denote by $2^S$
the simplex with vertex set $S$.
  We set $\deg x_i = \deg T = \deg z = 1$. 
By definition $\mathcal {H}$ is isomorphic
to the quotient of the polynomial ring $k[x_1, \dots , x_n, T, z]$
by a square-free monomial ideal.  
 We denote by $D_{\mathcal {H}}$ 
the simplicial complex on the vertex set  $ \{ 1,2, \dots ,n, T , z \}$
that corresponds to  the monomial ideal. The set of facets
of $D_{\mathcal {H}}$ is equal to the union 
\[
    \{  \{z,u \} : u  \text { facet of } D   \}  \; \; \cup  \; \; 
    \{    \{ T, 1,2, \dots, q+1,w \} :  w \text{ facet of } L \}.
\]
As a consequence, we have the following decomposition
\[
    D_{\mathcal {H}} = E_1 \cup E_2
\]
where  $E_1 = 2^{\{ z \}} * D$,
$E_2 = 2^{\{ T, 1, 2, \dots , q+1 \}} * L$.

Since  $D$ is Cohen--Macaulay over $k$ 
with dimension $d-1$ we have that $E_1$ 
 is Cohen--Macaulay over $k$ 
with dimension equal to $d$. 
Since  $L$ is Cohen--Macaulay over $k$ 
with dimension $d-q-1$ we have that  $E_2$ 
 is Cohen--Macaulay over $k$ 
with dimension equal to $d$.

Moreover,   $E_1 \cap E_2 = 2^{\{ 1,2 , \dots , q+1\}} * L$    is 
also  Cohen--Macaulay  over $k$  with dimension equal to $d-1$.
Hence using   \cite[Lemma 23.6]  {Hi} it follows
that $D_{\mathcal {H}}$ is  Cohen--Macaulay  over $k$  with dimension $d$.
Hence $\mathcal{H}$ is a Cohen--Macaulay ring with dimension equal to $d+1$.
\end {proof}

We denote by  ${\mathcal {H}}_a$ the ring
   ${\mathcal {H}}$   but with 
   $\deg x_i = \deg T = 1$ and $ \deg z = q$.  Since the dimension and the Cohen--Macaulay
property is independent of the grading, 
Proposition~\ref{prop!anputmathcalH} implies that 
   ${\mathcal {H}}_a$ is Cohen--Macaulay and $\dim {\mathcal {H}}_a =  d+1$.

The element $z-T^q \in {\mathcal {H}}_a$ is homogeneous and 
   ${\mathcal {H}}_a/(z-T^q)  \iso  R[T]/I_C$ as graded $k$-algebras.  Hence
$\dim R[T]/I_C \geq d+1-1 = d$.  By 
Lemma~\ref{lem!structureOfCno2} $C$ is Artinian.
Since $C = R[T]/(I_C, f_1 , \dots , f_d)$ it follows that
$\dim R[T]/I_C  \leq d$. As a consequence $\dim R[T]/I_C= d$.
Since  $z-T^q$ is homogeneous,
${\mathcal {H}}_a$ is Cohen--Macaulay and  
$\dim {\mathcal {H}}_a/ (z-T^q) = \dim {\mathcal {H}}_a -1$, it
follows that $z-T^q$ is  ${\mathcal {H}}_a$-regular.
Hence $R[T]/I_C$ is Cohen--Macaulay. As a consequence,
$f_1, \dots , f_d$ is an  $R[T]/I_C$-regular sequence.
Hence $ \HF ( R[T]/I_C) = \Gamma^d ( \HF (C))$.

Since  $I_G = (I, T I_L, T^{q+1}- x_{\sigma})$
it is clear that $I_C$ is  a subset of the initial ideal
of the $I_G$ with respect to the reverse lexicographic order
in $R[T]$ with $T > x_1 > \dots > x_n$. 
Since by Lemma~\ref{lem!structureOfCno2}
  $\HF(C)$ is equal to the Hilbert function of
a general Artinian reduction of $k[D_{\sigma}]$
and $k[D_{\sigma}]$ is Gorenstein of dimension $d$, it follows that
\[
    \HF(R[T]/I_C) =  \Gamma^d ( \HF (C)) =  \HF (k[D_{\sigma}]) = \HF(R[T]/I_G),
\]
with the last equality  by Lemma~\ref{lem!structureOfG}.
As a consequence,  $I_C$ is  equal to the initial ideal
of the $I_G$.  This finishes the proof of  Lemma~\ref{lem!aboutIC}.

\subsection {Proof of Lemma~\ref{lem!technicallemma3}} \label {subs!proofoflemma3}

By Lemma~\ref{lem!aboutAandB}  $A= \oplus_{i=0}^d A_i$ with $A_d$ $1$-dimensional and 
  $B= \oplus_{i=0}^{d-q-1} B_i$ with $B_{d-q-1}$   $1$-dimensional.
By Remark~\ref {rem!structureOfCSecondVersion} 
for all $i \geq 0$ we have
\begin {equation}  \label{eqn!structureOfC}
     C_i = A_i \oplus ( \oplus_{j=1}^q  T^j B_{i-j} )
\end {equation}
Hence $\dim C_{i} = \dim C_{d-i}$ for all $i \in  \ZZ$. In particular 
$C_{i} = 0$ for $i \geq d+1$ and  $C_{d}$ is $1$-dimensional.
Let $\omega \in R_1$ be a general linear element.  By
Lemma~\ref {lem!technical_lemma_57} 
$C$ has the WLP if and only if the multiplication by $\omega +T$
map $C_{p_1} \to C_{p_1+1}$ is injective and
the multiplication by $\omega +T$
map $C_{p_2} \to C_{p_2+1}$ is surjective.

We assume that $A$ has the WLP and $B$ satisfies property $M_{q,p_1}$ and we will
show that $C$ has the WLP.
For that we first show that the multiplication by $\omega +T$
map $C_{p_1} \to C_{p_1+1}$ is injective.
Assume it is not.    Then there exists   $0 \not= c \in C_{p_1}$ such that
\begin {equation}  \label{eqn!multiplbyciszero}
     (\omega + T) c = 0 
\end {equation}
 in $C_{p_1+1}$. 
By Equation~(\ref{eqn!structureOfC}) there exist (unique)
$a_{p_1} \in A_{p_1}$ and, for $1 \leq j \leq q$, $b_{p_1-j} \in B_{p_1-j}$ such that
\[
      c = a_{p_1} + \sum_{j=1}^q  b_{p_1-j} T^j 
\]
Since $A$ is assumed to have the  WLP, if $a_{p_1} \not= 0$ we have $\omega a_{p_1} \not= 0$
which implies     $ (\omega + T) c \not= 0$, 
which contradicts  Equation (\ref{eqn!multiplbyciszero}).     
Hence we have $a_{p_1} = 0$ in $A$.   Equation~(\ref{eqn!multiplbyciszero})
then implies that for $j=1,2, \dots , q-1$
\[
     \omega b_{p_1-1} = 0,    \quad \quad 
          b_{p_1-j} +  \omega b_{p_1-(j+1)} = 0, 
\]
with all equations in $B$.  Combining these equations we get 
$\omega^q b_{p_1-q} = 0$ in $B$. Using the assumption
that $B$ satisfies   property $M_{q,p_1}$  it follows that
$b_{p_1-q} = 0$ in $B$, which using the above equations
implies that $b_{p_1-j} = 0$ for all $1 \leq j \leq q$, hence
$c= 0$, a contradiction.

If $d$ is odd, since $p_1 = p_2$ and $\dim C_{p_1} = \dim C_{p_1+1}$
we get  that  the multiplication by $\omega +T$
map $C_{p_2} \to C_{p_2+1}$ is also surjective, hence 
by what we said above $C$ has the WLP.  If $d$ is even
we need the following argument:

Assume $c' \in C_{p_2+1}$ with
\[
   c'  =  a_{p_2+1} + \sum_{j=1}^q  b_{p_2+1-j} T^j 
\]
where $a_{p_2+1} \in A_{p_2+1}$ and  $b_{p_2+1-j}  \in
   B_{p_2+1-j}$ for all $1 \leq j \leq q$.
We will find
\[
   c  =  a_{p_2} + \sum_{j=1}^q  e_{p_2-j} T^j  \in C_{p_2}
\]
where $a_{p_2} \in A_{p_2}$ and  $e_{p_2-j}  \in
   B_{p_2-j}$ for all $1 \leq j \leq q$.
such that $ (\omega + T)  c = c'$. 
Hence we need to have
(with the first equation in $A$ and the remaining $q$ equations in $B$)
\begin {eqnarray*}
       a_{p_2+1}   &  = & \omega a_{p_2}  \\
       b_{p_2+1-1}  & = &   a_{p_2} + \omega e_{p_2-1} \\
       b_{p_2+1-2}  & = &   e_{p_2-1} + \omega e_{p_2-2} \\
                &  \vdots &  \\
       b_{p_2+1-q}  & = &   e_{p_2-q+1} + \omega e_{p_2-q} 
\end {eqnarray*}
\indent Since $A$ is assumed to have the WLP, 
the multiplication by $\omega$ map $A_{p_2} \to A_{p_2+1}$ is
surjective, hence there exists $a_{p_2} \in A_{p_2}$
such that   $a_{p_2+1}     =  \omega a_{p_2}$ in $A$. We fix such $a_{p_2}$. 
Given $e_{p_2-q} \in B_{p_2-q}$,  
the last $q-1$ equations in $B$ inductively determine (unique)
$e_{p_2-q+1},  \dots  , e_{p_2-1}$ that satisfy them.  What we need
is to choose $e_{p_2-q}$ in such a way that we have compatibility
with the second equation 
\[ 
       b_{p_2}   =    a_{p_2} + \omega e_{p_2-1}.
\]
If we express $e_{p_2-q+1},  \dots  , e_{p_2-1}$  in terms of $e_{p_2-q}$,
the compatibility equation becomes 
\[
     \omega^q e_{p_2-q}  = (-1)^{q} a_{p_2} + \sum_{i=0}^{q-1} (-1)^{q+1-i} \omega^i b_{p_2-i} 
\]
(equation in $B$).  Using the assumption
that $B$ satisfies   property $M_{q,p_1}$, 
Remark~\ref {rem!aboutMqpropertyno2} implies
that there exists $e_{p_2-q}  \in B_{p_2-q}$ that satisfies the compatibility
equation.  \\

We now prove the converse. We assume that $C$ has the WLP and
we prove that $A$ has the WLP and $B$ satisfies the $M_{q,p_1}$ property.
Since $C$ has the WLP,  as we said above the multiplication  by $\omega +T$
map $C_{p_1} \to C_{p_1+1}$ is injective and
the multiplication by $\omega +T$
map $C_{p_2} \to C_{p_2+1}$ is surjective.

Let $ a \in A_{p_2+1} \subset  C_{p_2+1}$.  Then there exists 
$c \in C_{p_2}$ such that $ a = (\omega + T) c$.
Write
\[
   c  =  a_{p_2} + \sum_{j=1}^q  e_{p_2-j} T^j  \in C_{p_2}
\]
with equality in $C$, 
where $a_{p_2} \in A_{p_2}$ and  $e_{p_2-j}  \in
   B_{p_2-j}$ for all $1 \leq j \leq q$.
It follows that $ a = \omega a_{p_2}$, 
hence the multiplication by $\omega$ map $A_{p_2} \to A_{p_2+1}$
is surjective.  
Using Remark~\ref{rem!aboutwlpforA} it follows
that $A$ has the WLP.

We will now prove that $B$ has the property $M_{q,p_1}$. We assume
it is not the case and we will get a contradiction.  By the assumption
there exists $0 \not= b \in B_{p_1-q}$ such that $\omega^q b = 0$
in $B$. We set $c =    \sum_{i=1}^q   (-1)^{q-i} \omega^{q-i} b T^{i} \in C_{p_1}$.
Since the  summand of $c$  corresponding to  $i=q$ is $bT^{q}$, 
we get from Equation~(\ref{eqn!structureOfC})  that 
$c \not= 0$. 
Using that $T^{q+1} = 0$ in $C$ we  have 
\begin{eqnarray*}
     (\omega + T) c  & = &  \sum_{i=1}^q   (-1)^{q-i} \omega^{q-i+1} b T^{i}
                               +  \sum_{i=1}^q   (-1)^{q-i} \omega^{q-i} b T^{i+1}  \\
                     & = &    
                 (-1)^{q-1} \omega^{q} b T +   \sum_{i=2}^q   (-1)^{q-i} \omega^{q-i+1} b T^{i}
                               +  \sum_{i=1}^{q-1}   (-1)^{q-i} \omega^{q-i} b T^{i+1} \\
          &   = &  0 
\end{eqnarray*}
\noindent which  is a contradiction, since the multiplication  by $\omega +T$
map $C_{p_1} \to C_{p_1+1}$ was assumed to be  injective. This contradiction 
finishes the proof of   Lemma~\ref{lem!technicallemma3}.

\subsection {Proof of Theorem~\ref{thm!forsDwlp} }    \label {subs!proofOfTHmforsDwlp}

We assume   $q > p_2 $  and $k[D_{\sigma}]$ has the WLP.
The assumption  that $k[D_{\sigma}]$ has the WLP implies by 
Lemma~\ref{lem!structureOfG} that $G$ has the WLP. 
 Moreover, by the same lemma we get
that $\HF (G) = \HF(C)$. As a consequence, 
Lemma~\ref{lem!structureOfCno2} implies that
$ G_d$ is $1$-dimensional
 and, $G_j = 0$ for $j > d$.
These properties, together with the Gorensteinness
(by Lemma~\ref{lem!structureOfG})
and the WLP of $G$ imply that for general 
$\omega \in R_1$ the multiplication by
$\omega +T$ map $G_{p_2} \to G_{p_2+1}$ is surjective.

Using that the map $\psi$ in
the proof of  Lemma~\ref{lem!structureOfG}  is bijective,  it follows that
the natural map $A \to G$, with $[a] \mapsto [a]$ for $a \in R$,
is injective.
Assume $a \in  A_{p_2+1} \subset  G_{p_2+1}$.  Then there 
exists $g \in G_{p_2}$ such that 
\begin {equation}   \label {eqn!forGasurjective}
      (\omega + T) g = a
\end {equation}
Using again that the map $\psi$ in
the proof of  Lemma~\ref{lem!structureOfG}  is bijective, 
there exists $e \in A_{p_2}$ and, for $1 \leq j \leq q$,
$b_{p_2-j} \in B_{p_2-j}$ such that
\[
    g =  e + \sum_{i=1}^q b_{p_2-i}T^i,
\]  
with equality in $G$.
The assumption $q > p_2$ implies that  $p_2-q < 0$, hence
$b_{p_2-q} =0$ in $B$.  As a consequence
\[
   a = (\omega + T) g = \omega e + \sum_{i=1}^{q-1} b_{p_2-i}T^i
                 + \sum_{i=1}^{q-1} b_{p_2-i}T^{i+1}
\]
(equality in $G$), which imply that $a = \omega e$ (equality in $A \subset G$). 
It follows that the multiplication by $\omega$ map $A_{p_2} \to A_{p_2 + 1}$ is surjective.
Using Remark~\ref{rem!aboutwlpforA} it follows
that $A$ has the WLP.
Hence,  Lemma~\ref{lem!technicallemma1}  implies that
$k[D]$ has the WLP. This finishes the
proof of Theorem~\ref{thm!forsDwlp}.

\section {Further results on the general question}  \label{sec!furtherresults}

In Corollary~\ref{cor!theiffstatement} 
we proved  that if  $\sigma$ is a face of a Gorenstein* simplicial complex $D$ and
$2 (\dim \sigma) > \dim D + 1$ then  the Stanley--Reisner ring
$k[D]$ has the WLP if and only if  $k[D_{\sigma}]$ has the WLP.
A natural question is whether the restriction on the dimension
of $\sigma$ is necessary. 
 Proposition~\ref {prop!aboutPaandPb},
which is motivated by  Propositions~\ref{prop!DiscussionAboutPaPb}  
and~\ref{prop!abouttripleequalityforilesd}, suggests some ideas that 
could, perhaps, prove useful  in  attacking this question.
In addition, in Remark~\ref {rem!secondproofofThmforhighq}
we give a second proof of  Corollary~\ref{cor!theiffstatement}
based on the constructions of the present section.

We continue using the definitions and notation of  Section~\ref{sec!wlpofstellars}.
We  set 
\[
    \mathcal{L} =  (I, f_1I_L, f_2, \dots , f_{d+1}) \subset R, \quad \quad
        {\mathcal{P}}_a =  \mathcal{L}: x_{\sigma} \subset R,
\]
\[
       {\mathcal{P}}_b = \mathcal{L} : f_1^{q+1} \subset R,
              \quad \quad
       {\mathcal{P}}_c = (  I_L ,  f_2, \dots , f_{d+1})  \subset R.
\]
We also define the ideal $I_Q = ( I, TI_L) \subset R[T]$ and
set $h_1 = \HF (k[D]),  h_2 = \HF (R/I_L)$, 
${\mathcal{K}}_a = \mathcal{L} + (x_{\sigma}) \subset R$, 
${\mathcal{K}}_b = \mathcal{L} + (f_1^{q+1}) \subset R$. There are two 
short exact sequences
\begin {equation}  \label{eqn!shortexactseq1forPa}
   0  \to (R/{\mathcal{P}}_a) (-q-1) \to R/\mathcal{L} \to  R/{\mathcal{K}}_a  \to 0
\end {equation}
and
\begin {equation}  \label{eqn!shortexactseq2forPb}
   0  \to (R/{\mathcal{P}}_b) (-q-1) \to R/\mathcal{L} \to  R/{\mathcal{K}}_b  \to 0
\end {equation}
where $(-q-1)$  denotes twist by $-q-1$, the map $ R/{\mathcal{P}}_a \to  R/\mathcal{L}$
is multiplication by $x_{\sigma}$ and the map 
$  R/{\mathcal{P}}_b  \to R/\mathcal{L}$
is multiplication by $f_1^{q+1}$.

We also have
\begin {equation}  \label  {eqn!formathcalKa}
     R/{\mathcal{K}}_a  =  k[D_{\sigma}]/( T-f_1, f_2, \dots , f_{d+1}),
\end {equation}
and
\begin {equation}  \label  {eqn!formathcalKb}
     R/{\mathcal{K}}_b  =  (R[T]/I_C)/( T-f_1, f_2, \dots , f_{d+1}).
\end {equation}

The following proposition gives the main properties of $R[T]/I_Q$.

\begin {proposition}  \label{prop!infoaboutIQ}
  i)  We have  $ \dim R[T]/I_Q =  d+1 $ and $\depth R[T]/I_Q = d$.

 ii)  For all $m \geq 0$ we have  
     \[
              \HF( m,  R[T]/I_Q) =   \HF(m, k[D] ) + \sum_{j=1}^{m} \HF (m-j, R/I_L)
     \]

iii)  If  $t \leq d$ then the sequence $T-f_1, f_2, \dots , f_t$ is $R[T]/I_Q$-regular.
\end {proposition}

\begin {proof}
i)  Denote by   $D_Q$ the simplicial complex
that corresponds to the square-free
monomial ideal $I_Q$.
We set $P_1 = (I,T) \subset R[T] , P_2 = (I_L) \subset R[T]$.
  Using that $I \subset I_L$ and that $T$ is a new variable we have
 $   I_Q  =  P_1 \cap P_2 $.
Hence
\[
    D_Q =  D  \; \cup   \;  (   2^{\{T, 1, \dots , q+1\}} * L ).
\]
As a consequence,
since $\dim D =  d-1$ and  $\dim L = d-q-2$  we have
$\dim D_Q = d$, hence $\dim R[T]/I_Q = d+1$. Since
$D_Q$ is not pure  (in the sense that it contains facets of
different dimensions), by 
\cite [Corollary 5.1.5]  {BH}
  $R[T]/I_Q$ is not Cohen--Macaulay,
hence $\depth R[T]/I_Q \leq d$.

We have that $R[T]/P_1$ is Cohen--Macaulay of dimension $d$, hence
has depth $d$,
$R[T]/P_2$ is Cohen--Macaulay of dimension $d+1$, hence has depth
$d+1$ and $R[T]/(P_1+P_2)$ is Cohen--Macaulay of dimension
$d$, hence has depth $d$.    Using Equation~(\ref{eqn!depthwithExts})
and the additivity of the Ext functor on the second variable
(\cite [Proposition~~3.3.4]{Wei})
we get that $\depth (R[T]/P_1 \oplus R[T]/P_2) = d$.

Since $I_Q = P_1 \cap P_2$,  there is, by 
\cite [Proof of Theorem 5.1.13]  {BH},  a short exact sequence
of $R$-modules
\begin{equation}  \label {eqn!shortExactSequenceno1}
    0 \to R[T]/ I_Q \to  R[T]/P_1 \oplus R[T]/P_2 \to R[T]/(P_1+P_2) \to 0
\end {equation}
As a consequence, using  \cite [Proposition~1.2.9]  {BH}
we get 
\[
      \depth R[T]/I_Q  \; \geq  \;  \depth (R[T]/P_1 \oplus R[T]/P_2) =  d.
\]
   Since we proved above that
$\depth R[T]/I_Q \leq d$, we get $\depth R[T]/I_Q = d$.

ii)  Assume $m \geq 0$.
  Using the short exact sequence (\ref{eqn!shortExactSequenceno1})
 we have
  \[
      \HF ( m, R[T]/I_Q ) =  \HF (m, R[T]/P_1) + \HF (m, R[T]/P_2) -
                \HF (m, R[T]/(P_1+P_2))
\]      
Since $\HF (m, R[T]/P_1) = h_1(m)$, 
$\HF(m,R[T]/P_2)= (\Gamma^1 (h_2)) (m) = \sum_{j=0}^m h_2(j)$ and
$\HF(m,R[T]/ (P_1+P_2))=   h_2(m)$ the result follows.

iii) Since  $\depth R[T]/I_Q = d$,  $t \leq d$,  and the ideal of $R[T]$ generated
by    $T-f_1, f_2, \dots , f_t$  is equal to the ideal of $R[T]$ generated
by $t$ general linear elements,  Lemma~\ref {lem!aboutGeneralLinearsAndDepth} 
implies that the sequence $T-f_1, f_2, \dots , f_t$ is $R[T]/I_Q$-regular.
\end {proof}

The following two propositions motivate  Proposition~\ref {prop!aboutPaandPb}.

\begin {proposition}  \label{prop!DiscussionAboutPaPb}  
For all $m \geq 0$ we have 
\begin{equation}   \label{eqn!conca_inequality_for_Pb_pa}
         \HF (m, R/{\mathcal{P}}_b) \leq \HF (m, R/{\mathcal{P}}_a).
\end {equation}
\end {proposition}
 
\begin {proof}
   Since $f_1, \dots  , f_{d+1}$ are  general linear elements of $R$ it follows that
  the ideal of $R[T]$ generated by    $T-f_1, f_2, \dots , f_{d+1}$  is equal to the ideal of 
 $R[T]$ generated by $d+1$ general linear elements.
  Since by the proof of Lemma~\ref{lem!technicallemma2}   
  $I_C$ is an initial ideal of $J_{st}$ (with respect to a suitable monomial order)
  by  \cite [Theorem~1.1]{Con} we have 
\[
    \HF (m, k[D_{\sigma}]/(T-f_1, f_2, \dots , f_{d+1}))   \leq 
         \HF (m, R[T]/(I_C, T-f_1, f_2, \dots , f_{d+1}))  
\]
for all $m \geq 0$.  Hence, Equations~(\ref{eqn!formathcalKa}) and
(\ref{eqn!formathcalKb}) imply that   
\[
    \HF (m, R/{\mathcal{K}}_a )   \leq 
         \HF (m, R/{\mathcal{K}}_b)  
\]
  for all $m \geq 0$. As a consequence, the short exact sequences 
(\ref{eqn!shortexactseq1forPa}) and (\ref{eqn!shortexactseq2forPb})
imply  Inequality~(\ref{eqn!conca_inequality_for_Pb_pa}).   \end {proof}

\begin {proposition}  \label{prop!abouttripleequalityforilesd}
 Assume $1 \leq t \leq d$. Then we have the following equality of ideals of $R$
\[
    ( I, f_1 I_L ,  f_2, \dots , f_t) : x_{\sigma} = 
    ( I, f_1 I_L ,  f_2, \dots , f_t) : f_1^{q+1} =
           (  I_L ,  f_2, \dots , f_t) .
\]
\end {proposition}

\begin {proof} 
Recall $h_2 = \HF (R/I_L)$.
We first prove the equality for $t=1$.
For simplicity we set $\; h_3 =   \HF  (R[T]/I_Q), \;  \;  h_4 = \HF (k[D_{\sigma}])$,
\[
    J_1 =  ( I, f_1 I_L ) : x_{\sigma},  \quad  \quad
    J_2 =  ( I, f_1 I_L ) : f_1^{q+1}.
\]

Since  it is clear that, for $i=1,2$, we have
  $I_L \subset  J_i$, to prove $J_1 = J_2 = I_L$
 it is enough to   prove that   
\[
     \HF (R/J_1) = \HF (R/ J_2) = h_2.
\] 
Consider the short exact sequence
\begin {equation}   \label{eqn!secusaqr3}
    0 \to  (R/J_1) (-q-1)
          \to     R[T]/(I_Q, T-f_1)
              \to  R[T]/(I_Q, T-f_1, x_{\sigma})   \to 0
\end {equation}
 where  $(-q-1)$ means twist  by $-q-1$ and the first map
 is multiplication by $x_{\sigma}$.
We have $R[T]/(I_Q, x_{\sigma}) = k[D_{\sigma}]$.
Moreover,  the ideal of $R[T]$ generated by   $T-f_1$ is equal to the ideal 
generated by a general linear element of $R[T]$. Since $k[D_{\sigma}]$
is Gorenstein of dimension $d$ it follows from Lemma~\ref{lem!aboutGeneralLinearsAndDepth}
that $T-f_1$ is $k[D_{\sigma}]$-regular, hence
\[
    \HF( R[T]/(I_Q, T-f_1, x_{\sigma})) = \Delta^1 (h_4)
\]
Since by  Proposition~\ref{prop!infoaboutIQ} $\depth R[T]/I_Q = d$,
it follows from Lemma~\ref{lem!aboutGeneralLinearsAndDepth}
that $T-f_1$ is $R[T]/I_Q$-regular, hence
\[
    \HF( R[T]/(I_Q, T-f_1)) = \Delta^1 (h_3).
\]
As a consequence,  the short exact sequence (\ref{eqn!secusaqr3})
implies that
\begin{equation}    \label{eqn!aboutHFandDelta}
    \HF ((R/J_1) (-q-1)) = \Delta^1 ( h_3 - h_4)
\end {equation}

Combining the computation of
$h_3$ in  Proposition~\ref {prop!infoaboutIQ}
and  the computation of  $h_4$ in Lemma~\ref{lem!aboutAandB}.
we get that,  for all $m \geq 0$,
\[
    (h_3 - h_4) (m) = \sum_{j=q+1}^{m} h_2 (m-j) =
           \sum_{j=0}^{m-q-1} h_2 (j) 
\]
Hence for all $m \geq 0$  
\[
    (\Delta^1 (h_3-h_4) )(m)  =  h_2 (m-q-1).
\] 
As a consequence,  Equation~(\ref{eqn!aboutHFandDelta}) implies that
$\HF (m,R/J_1) = h_2(m)$ for all $m \geq 0$. 

Consider now the short exact sequence
\begin {equation}   \label{eqn!sndjqu2}
    0 \to  (R/J_2) (-q-1)
          \to     R[T]/(I_Q, T-f_1)
              \to  R[T]/(I_Q, T-f_1, T^{q+1})   \to 0
\end{equation}
We have $R[T]/(I_Q, T^{q+1}) = R[T]/I_C$,
which by Lemma~\ref{lem!aboutIC} is Cohen--Macaulay with same
Hilbert function  as $k[D_{\sigma}]$.  Since  the ideal of $R[T]$ generated by 
  $T-f_1$ is equal to  the ideal generated by a
general linear element of $R[T]$, we get
\[
    \HF( R[T]/(I_Q, T-f_1, f_1^{q+1})) = \Delta^1 (h_4).
\]
As a consequence, using the previously done computations of
$\HF ( R[T]/(I_Q, T-f_1))$ and     $\Delta^1 (h_3-h_4) $,
the short exact sequence   (\ref{eqn!sndjqu2})  implies
that $\HF (m,R/J_2) = h_2(m)$ for all $m \geq 0$. 
This finishes the proof of the double equality 
$J_1 = J_2 =I_L$.

We now assume $2 \leq t \leq d$.
  We set 
\[
    J_4=  ( I, f_1 I_L ,  f_2, \dots , f_t) : x_{\sigma},  \quad  \quad
    J_5 =  ( I, f_1 I_L ,  f_2, \dots , f_t) : f_1^{q+1},  
\]
\[
        J_6 =  (  I_L ,  f_2, \dots , f_t) .
\]

Clearly, for $i=4,5$, we have that $J_6 \subset  J_i$. Hence, 
to prove $J_4=J_5= J_6$ it is
enough to prove that  $\HF(R/J_4) = \HF (R/J_5) = \HF (R/J_6)$.
To prove this equality, it is enough to prove that
\[
    \HF(R/J_4) = \Delta^{t-1} \HF(R/J_1), \quad 
    \HF(R/J_5) = \Delta^{t-1} \HF(R/J_2),
\]
and
\[
    \HF(R/J_6) = \Delta^{t-1} \HF(R/I_L), 
\]
since, as we proved above, $\HF(R/J_1) = \HF (R/J_2) = \HF (R/I_L)$.

By Lemma~\ref{lem!aboutAandB}   $R/I_L$   is Gorenstein with $\dim R/I_L = d$. Since
$t \leq d$ and $f_2, \dots , f_t$ are general linear elements over
an infinite field, it follows by Lemma~\ref{lem!aboutGeneralLinearsAndDepth}
that they are a regular $R/I_L$ sequence,
hence  $    \HF(R/J_6) = \Delta^{t-1} \HF(R/I_L)$.

For $J_4$ we have  the short exact sequence
\[
  0 \to   R/J_4  (-q-1) \to   R/(I, f_1I_L, f_2, \dots ,f_t) \to
                     R/(I, f_1I_L, f_2, \dots f_t, x_{\sigma}) \to 0
\]
while for $J_5$ we have  the short exact sequence
\[
  0 \to   R/J_4  (-q-1) \to   R/(I, f_1I_L, f_2, \dots ,f_t) \to
                     R/(I, f_1I_L, f_2, \dots f_t,  f_1^{q+1}) \to 0
\]
Hence, it is enough to prove the following three equalities
\[
   \HF (R/(I, f_1I_L, f_2, \dots ,f_t))  = 
           \Delta^{t-1}   \HF (R/(I, f_1I_L)),
\]
\[
   \HF (R/(I, f_1I_L, f_2, \dots ,f_t , x_{\sigma}))  = 
         \Delta^{t-1}   \HF (R/(I, f_1I_L, x_{\sigma})),
\]
and
\[
   \HF (R/(I, f_1I_L, f_2, \dots ,f_t , f_1^{q+1}))  = 
         \Delta^{t-1}   \HF (R/(I, f_1I_L, f_1^{q+1})).
\]

Each of the three  equalities follows using
Lemma~\ref{lem!aboutGeneralLinearsAndDepth},
taking into account that $t \leq d$,   
the ideal $(T-f_1,f_2, f_3, \dots , f_t)$ of $R[T]$  
is equal to the ideal of $R[T]$ generated by $t$
general linear elements, and that 
$\depth R[T]/I_Q = d$  (by  Proposition~\ref{prop!infoaboutIQ}),
$\depth k[D_{\sigma}] = d$  (since 
$k[D_{\sigma}]$ is Gorenstein of dimension  $d$) and
$\depth R[T]/I_C = d$  (since 
by Lemma~\ref{lem!aboutIC}
$R[T]/I_C$ is Cohen--Macaulay of dimension  $d$).
\end {proof}

\begin {remark} 
As we mention below in Remark~\ref{rem!aboutPropForPaPbno2},
for $t= d+1$, which is the critical value for the
WLP properties, the triple equality of ideals 
in the statement of Proposition~\ref{prop!abouttripleequalityforilesd}
does not hold any more.
\end {remark}

For a homogeneous ideal $J \subset R$ and $t \geq 0$ 
we denote by $J_{\leq t}$ 
the ideal of $R$ generated by all homogeneous elements of $J$ 
that have degree $\leq t$. The following 
proposition is motivated by  Propositions~\ref{prop!DiscussionAboutPaPb}  
and~\ref{prop!abouttripleequalityforilesd}.

\begin {proposition}  \label {prop!aboutPaandPb}
i)  Assume $k[D_{\sigma}]$ has the WLP.  
     If $\HF(R/{\mathcal{P}}_a) = \HF (R/{\mathcal{P}}_b)$ then
     $k[D]$ has the WLP.    

ii)   Assume $k[D_{\sigma}]$ has the WLP and condition $M_{q,p_1}$ holds for $k[L]$.
     Then             $k[D]$ has the WLP if and only if 
               $\HF(R/{\mathcal{P}}_a) = \HF (R/{\mathcal{P}}_b)$.

iii) Assume   $k[D_{\sigma}]$  has the WLP and $d$ is odd. 
     Assume   the following equality of ideals of $R$ holds:
    \[
        ({\mathcal{P}}_b)_{\leq p_1-q} = ({\mathcal{P}}_c)_{\leq p_1-q}.
   \]
   Then $k[D]$ has the WLP.  
\end {proposition}

\begin {proof}
We  first prove i).
The short exact sequences (\ref{eqn!shortexactseq1forPa}) and (\ref{eqn!shortexactseq2forPb})
imply that $\HF (R/{\mathcal{K}}_a) = \HF (R/{\mathcal{K}}_b)$  if and only if
 $\HF(R/{\mathcal{P}}_a) = \HF (R/{\mathcal{P}}_b)$.
By Equation~(\ref{eqn!formathcalKa})
\[
     R/{\mathcal{K}}_a  =  k[D_{\sigma}]/( T-f_1, f_2, \dots , f_{d+1}). 
\] 
Since  the ideal  $(T-f_1, \dots , f_{d+1})$ is equal to the ideal
of $R[T]$ generated by $d+1$ general linear elements, 
the assumption that $k[D_{\sigma}]$ has the WLP
and the fact  (see Lemma~ \ref{lem!structureOfCno2})
that  $\HF(C) $ is equal to the HF of a general Artinian 
reduction of $k[D_{\sigma}]$ imply that
\[
    \HF (R/{\mathcal{K}}_a) = \Delta^+ (\HF (C)).
\]
Hence from  the assumption $\HF(R/{\mathcal{P}}_a) = \HF(R/{\mathcal{P}}_b)$ it follows that
\[
    \HF (R/{\mathcal{K}}_b) =   \HF (R/{\mathcal{K}}_a) = \Delta^+ (\HF (C))  
\]  
hence $C$ has the WLP.  Lemma~\ref{lem!technicallemma3} 
 implies that $A$ has the WLP. Hence, 
by  Lemma~\ref {lem!technicallemma1} $k[D]$ has the WLP.

ii) We  assume that $k[D_{\sigma}]$ has the WLP, condition $M_{q,p_1}$ holds for $k[L]$
   and    $\HF(R/{\mathcal{P}}_a) = \HF (R/{\mathcal{P}}_b)$.
    From Part i) of the present proposition we get that     
                 $k[D]$ has the WLP.

   We now  assume that both 
   $k[D_{\sigma}]$ and $k[D]$ have the WLP and  condition $M_{q,p_1}$ holds for $k[L]$ 
   and we will prove that 
    $\HF(R/{\mathcal{P}}_a) = \HF (R/{\mathcal{P}}_b)$.
    Combining    Lemmas~\ref{lem!technicallemma1} and
    \ref{lem!technicallemma3}  we get that $C$ has the WLP.
    Hence both $k[D_{\sigma}]$ and $C$ have the WLP.
    Since by  Lemma~ \ref{lem!structureOfCno2}
    $\HF(C) $ is equal to the HF of a general Artinian 
    reduction of $k[D_{\sigma}]$, we get that 
    $\HF (R/{\mathcal{K}}_b) =   \HF (R/{\mathcal{K}}_a)$
    which, 
   using the two short exact sequences (\ref{eqn!shortexactseq1forPa})
   and  (\ref{eqn!shortexactseq2forPb}),
   implies    $\HF(R/{\mathcal{P}}_a) = \HF (R/{\mathcal{P}}_b)$.

 iii)    Since $d$ is odd we have $p_1 = p_2 = (d-1)/2$.   Since always ${\mathcal{P}}_c \subset {\mathcal{P}}_a$, we get that 
     for all $m \geq 0$   
    \[
               \HF (m, R/{\mathcal{P}}_a)  \leq \HF (m,  R/{\mathcal{P}}_c).
    \]
    Moreover,  by Proposition~\ref {prop!DiscussionAboutPaPb} we have 
     that    $\HF (m, R/{\mathcal{P}}_b) \leq \HF (m, R/{\mathcal{P}}_a)$ for all $m \geq 0$. 
    Using  that $p_1=p_2$ and the assumption 
   $({\mathcal{P}}_b)_{\leq p_1-q} = ({\mathcal{P}}_c)_{\leq p_1-q}$ we get
    \[
               \HF (i , R/{\mathcal{P}}_a)  =  \HF (i,  R/{\mathcal{P}}_b)
    \] 
    for all $i \leq p_2-q$.  As a consequence,
    the two short exact sequences (\ref{eqn!shortexactseq1forPa})
   and (\ref{eqn!shortexactseq2forPb})  imply that 
   \begin {equation}   \label{eqn!hfno2no3no4}
                   \HF (i , R/{\mathcal{K}}_a)  =  \HF (i,  R/{\mathcal{K}}_b)
   \end {equation}
   for all $ i \leq  p_2-q + q+1 = p_2 +1$.   The assumption that
   $k[D_{\sigma}]$ has the WLP  gives $\HF (p_2+1, R/{\mathcal{K}}_a) = 0$.
   Hence, Equation~(\ref{eqn!hfno2no3no4})  implies
   $\HF (p_2+1, R/{\mathcal{K}}_b)= 0 $.  As a consequence, 
   Equation~(\ref{eqn!hfno2no3no4}) implies
   $\HF (R/{\mathcal{K}}_a) = \HF (R/{\mathcal{K}}_b)$.  Hence,
   Corollary~ \ref{corol!HFofsectionofC} implies that $C$ has the WLP,
   which by Lemma~\ref{lem!technicallemma3} implies that $A$ has the WLP
   which by Lemma~\ref{lem!technicallemma1} implies that
   $k[D]$ has the WLP.   
\end {proof}

\begin {remark}  \label {rem!aboutPropForPaPbno2}
  It is clear that  we always have ${\mathcal{P}}_c \subset {\mathcal{P}}_a \cap {\mathcal{P}}_b$.
   Macaulay2 \cite{GS} computations for some $1$-faces
  of the boundary complex of the cyclic
  polytope with $10$ vertices and dimension $6$    
  give examples such that   ${\mathcal{P}}_a,{\mathcal{P}}_b$ have
   the same HF but are not equal as ideals of $R$
 and, moreover, the ideal ${\mathcal{P}}_c$ is a proper  subset of ${\mathcal{P}}_a \cap {\mathcal{P}}_b$. 
\end {remark}

\begin {remark}   The present remark is related to Part iii) of  Proposition~\ref {prop!aboutPaandPb}.
   Macaulay2   computations suggests  that independently of whether 
   $d$ is even or odd  the inequality  $({\mathcal{P}}_b)_{\leq p_1-q} = ({\mathcal{P}}_c)_{\leq p_1-q}$ 
   perhaps holds.   Can it be proven theoretically?  However,  when $d$ is even it seems
  that even if we assume that  $k[D_{\sigma}]$ has the WLP and
  $({\mathcal{P}}_b)_{\leq p_1-q} = ({\mathcal{P}}_c)_{\leq p_1-q}$ 
  it is not clear  how to conclude that $k[D]$ has the WLP.
\end  {remark}

\begin {remark}   \label {rem!secondproofofThmforhighq}
 Assume $q > p_2$. 
We give a second proof
 of  Corollary~\ref{cor!theiffstatement}
 that does not  use the ring $R[T]/I_G$.
The two short exact sequences (\ref{eqn!shortexactseq1forPa})
and (\ref{eqn!shortexactseq2forPb})  imply that 
$\HF (i, R/{\mathcal{K}}_a) = \HF (i, R/{\mathcal{K}}_b)$ for all $0 \leq i \leq q$.
By Equation~(\ref{eqn!formathcalKa})
$R/{\mathcal{K}}_a$ is isomorphic to a general 
linear section of a general Artinian reduction of $k[D_{\sigma}]$,
while  by Equation~(\ref{eqn!formathcalKb})
$R/{\mathcal{K}}_b$ is isomorphic to a general 
linear section  of $C$. 

 Assume first  that $k[D_{\sigma}]$ has the WLP,
then $\HF (p_2+1, R/{\mathcal{K}}_a) = 0$. Since $q > p_2$ we get
that $\HF (p_2+1, R/{\mathcal{K}}_b)= 0 $.  As a consequence, 
$\HF (R/{\mathcal{K}}_a) = \HF (R/{\mathcal{K}}_b)$.  Hence,
Corollary~ \ref{corol!HFofsectionofC} implies that $C$ has the WLP,
which by Lemma~\ref{lem!technicallemma3} implies that $A$ has the WLP
which by Lemma~\ref{lem!technicallemma1} implies that
$k[D]$ has the WLP.

Conversely assume that  $k[D]$ has the WLP. By Lemma~\ref{lem!technicallemma1}
$A$ has the WLP. Since $q > p_2$  and $p_2 \geq p_1$ by Lemma~\ref{lem!aboutMqforhighq}
property $M_{q,p_1}$ holds for $B$.  Hence Lemma~\ref{lem!technicallemma3}
implies that $C$ has the WLP. Then $\HF (p_2+1, R/{\mathcal{K}}_b) = 0$. 
Since $q > p_2$ we get that $\HF (p_2+1, R/{\mathcal{K}}_a) = 0$.  As a consequence, 
$\HF (R/{\mathcal{K}}_a) = \HF (R/{\mathcal{K}}_b)$. 
Hence, Corollary~\ref{corol!HFofsectionofC} implies 
that $k[D_{\sigma}]$ has the WLP. This finishes the second proof
 of  Corollary~\ref{cor!theiffstatement}.
\end {remark}

\subsection {An interesting example not coming from simplicial 
        complexes}  \label{subs!useful_counterexample}

The following example   is related to   Remark~\ref{rem!aboutUsingG} below
and is  taken from \cite[Example 2.6]{ScSe}.
Assume $k$ is an infinite field of characteristic $0$
and $R=k[x_1, \dots , x_4]$
with the degrees of all variables equal to $1$.
   Consider the homogeneous ideal
\[
   I = (x_1^3, x_2^3, x_3^3, x_4^3, (x_1+x_2+x_3+x_4)^3 ) \subset R
\]
The ideal
$I$ is an Artinian  (hence Cohen--Macaulay) codimension $4$ almost complete intersection without
the WLP.   We set $I_L = (I : x_1x_2) \subset R$. 
Macaulay2 computations  suggest that  the Artinian ideal
$(I ,  T^2, TI_L) \subset  R[T]$  does not have the WLP, while the Artinian ideal
$(I , T^2-x_1x_2, TI_L) \subset  R[T]$
does  have the WLP.   
 
\subsection {Final Remarks}  \label{subs!further_discussions23}

In the following remarks we  keep assuming that 
$D$ is a Gorenstein* simplicial complex and $k$ is an infinite field.

\begin {remark}  
Suppose 
\[
     D_0 = D, D_1,  \dots  ,  D_m
\]
is a finite sequence of simplicial complexes  such that,
for all $0 \leq i \leq m-1$, the complex $D_{i+1}$ is obtained 
from $D_{i}$ by a stellar subdivision with respect to
a face $\sigma_{i}$ of $D_{i}$  with
$2(\dim \sigma_{i} ) > \dim D +1$.   Then,  
by Corollary~\ref{cor!theiffstatement} the Stanley--Reisner ring  $k[D]$ 
has the WLP
if and only if $k[D_m]$ has the WLP.  Is it possible 
to prove that starting from $D$  there exists a sequence of stellar subdivisions
as above with $k[D_m]$ WLP?  Then it would follow that 
$k[D]$ has the WLP. Compare also \cite  [Conjecture~4.12]{KuNe}.
\end {remark}

\begin {remark}  Recall $I_C = (I,  T^{q+1},  T I_L )$. 
Assume  that $k[D_{\sigma}]$ has the WLP.
Is it possible to prove that  $R[T]/I_C$ has the WLP?   If so, combining 
Lemmas~\ref{lem!technicallemma1} and  \ref{lem!technicallemma3}
it would then follow   that  $k[D]$ has the WLP.
\end {remark}

\begin {remark}    \label {rem!aboutUsingG}
 Recall $I_G = (I,  T^{q+1}- x_{\sigma},  T I_L )$,
$G=R[T]/(I_G+(f_1, \dots ,f_d))$.  Assume  $k[D_{\sigma}]$ has the  WLP. 
Then by Lemma~\ref{lem!structureOfG}
$R[T]/I_G$ and $G$ have the WLP.  But it is not clear how to use that to prove
that $k[D]$  has the WLP.   Compare also the example 
in Subsection~\ref{subs!useful_counterexample} which does not
come from simplicial complexes.

In general  $G$ WLP  only implies      $\HF (G/(f_{d+1}+T)) = \Delta^+ (\HF (G)$).
However, some Macaulay2 computational evidence suggests that perhaps in the Gorenstein* setting
it holds  that
\begin{equation}   \label{eqn!conjecturealForLefGa}
          \HF (G/(f_{d+1})) = \Delta^+ (\HF (G)).
\end{equation}
Is it possible to theoretically prove Equation (\ref{eqn!conjecturealForLefGa})?
Assume Equation  (\ref{eqn!conjecturealForLefGa}) holds. Then the  multiplication
by $f_{d+1}$ map $G_{p_1} \to G_{p_1+1}$ is injective.
Since by the proof of   Lemma~\ref{lem!structureOfG}
$ A_i \subset G_i$ for all $i$, 
we get that   the  multiplication
by $f_{d+1}$ map $A_{p_1} \to A_{p_1+1}$ is injective.
By Remark~\ref{rem!aboutwlpforA} $A$ has the WLP, hence by   
Lemma~\ref{lem!technicallemma1} $k[D]$ has the WLP.
\end {remark}

\end{document}